\documentclass[a4paper,11pt]{amsart}
\usepackage{mathrsfs}
\usepackage{appendix}
\usepackage{amssymb}
\usepackage{fancyhdr}
\usepackage{charter}
\usepackage{typearea}
\usepackage{color}
\usepackage{amsmath,amstext,amsthm,amscd}
\usepackage{url}

\usepackage[a4paper,top=3cm,bottom=3cm,left=3cm,right=3cm]{geometry}

\makeatletter
      \def\@setcopyright{}
      \def\serieslogo@{}
      \makeatother

\newcommand{\Complex}{\mathbb C}
\newcommand{\Real}{\mathbb R}
\newcommand{\N}{\mathbb N}
\newcommand{\ddbar}{\overline\partial}
\newcommand{\pr}{\partial}
\newcommand{\ol}{\overline}

\newcommand{\norm}[1]{\left\Vert#1\right\Vert}
\newcommand{\abs}[1]{\left\vert#1\right\vert}

\newcommand{\set}[1]{\left\{#1\right\}}
\newcommand{\To}{\rightarrow}

\theoremstyle{plain}
\newtheorem{theorem}{Theorem}[section]
\newtheorem{lemma}[theorem]{Lemma}

\newtheorem{proposition}[theorem]{Proposition}

\newtheorem{definition}[theorem]{Definition}

\newtheorem{ass}[theorem]{Assumption}

\numberwithin{equation}{section}

\begin{document}
\title[{Torus equivariant Szeg\H{o} kernel asymptotics on strongly pseudoconvex CR manifolds}]
{Torus equivariant Szeg\H{o} kernel asymptotics on strongly pseudoconvex CR manifolds}
\author[Hendrik Herrmann]{Hendrik Herrmann}
\address{Mathematical Institute, University of Cologne, Weyertal 86-90, 50931 Cologne, Germany}
%\thanks{Hendrik Herrmann was partially supported by the CRC TRR 191: ``Symplectic Structures in Geometry, Algebra and Dynamics''. He would like to thank  the Mathematical Institute, Academia Sinica, and the School of Mathematics and Statistics, Wuhan University, for hospitality, a comfortable accommodation and financial support during his visits in January and March - April, respectively.}
\email{heherrma@math.uni-koeln.de or post@hendrik-herrmann.de}
\author[Chin-Yu Hsiao]{Chin-Yu Hsiao}
\address{Institute of Mathematics, Academia Sinica and National Center for Theoretical Sciences, Astronomy-Mathematics Building, No. 1, Sec. 4, Roosevelt Road, Taipei 10617, Taiwan}
\thanks{Chin-Yu Hsiao was partially supported by Taiwan Ministry of Science and Technology project 106-2115-M-001-012 and Academia Sinica Career Development
Award. }
\email{chsiao@math.sinica.edu.tw or chinyu.hsiao@gmail.com}
\author[Xiaoshan Li]{Xiaoshan Li}
\address{School of Mathematics
and Statistics, Wuhan University, Wuhan 430072, Hubei, China}
\thanks{Xiaoshan Li was supported by  National Natural Science Foundation of China (Grant No. 11501422).}
%\thanks{The second author was  supported by Central university research Fund 2042015kf0049, Postdoctoral Science Foundation of China 2015M570660 and NSFC No. 11501422}
%\thanks{The second-named author was partially supported by Central university research Fund 2042015kf0049.}
\email{xiaoshanli@whu.edu.cn}
%\author[Hendrik Herrmann]{Hendrik Herrmann}
%\address{School of Mathematics
%and Statistics, Wuhan University, Hubei 430072, China}
\begin{abstract}
Let $(X, T^{1,0}X)$ be a compact strongly pseudoconvex CR manifold of dimension $2n+1$. Assume that $X$ admits a Torus action $T^d$. In this work,  we study the behavior of torus equivariant Szeg\H{o} kernels and prove that
the weighted  torus equivariant Szeg\H{o} kernels admit asymptotic expansions.
\end{abstract}

\maketitle \tableofcontents

\section{Introduction and statement of the main results}\label{s-gue170805}

Let $(X,T^{1,0}X)$ be a compact strongly pseudoconvex CR manifold of dimension $2n+1$, $n\geq1$. Assume that $X$ admits a compact connected Lie group action $G$. The study of $G$-equivariant CR functions and Szeg\H{o} kernel is closely related to some problems in CR, complex  geometry, Mathematical physics and geometric quantization theory. For example, for a compact irregular Sasakian manifold $X$, it was shown in~\cite{HHL17} that $X$ admits a torus action $T^d$ and the study of torus-equivariant CR functions and Szeg\H{o} kernel is important in Sasaki geometry. In this work, we consider a compact strongly pseudoconvex CR manifold $(X, T^{1,0}X)$ of dimension $2n+1$ and assume that $X$ admits a  torus
action $T^d=(e^{i\theta_1},\ldots,e^{i\theta_d})$. We introduce  the weighted  torus equivariant Szeg\H{o} kernels $S_{k,\tau}(x)$ and $S_k(x)$ (see\eqref{e-gue180528a}, \eqref{e-gue180529mIII}). We show that the weighted Szeg\H{o} kernel $S_{k,\tau}(x)$ admits a full asymptotic expansion (see Theorem~\ref{t-gue180528}) and we obtain the asymptotic leading term of the weighted Szeg\H{o} kernel $S_{k}(x)$ (see Theorem~\ref{t-gue180528I}).

We now formulate our main results. We refer the reader to Section~\ref{s:prelim} for some standard notations and terminology used here.
Let $(X,T^{1,0}X)$ be a compact strongly pseudoconvex CR manifold of diemension $2n+1$, $n\geq1$.
In this work, we assume that $X$ admits a torus
action $T^d=(e^{i\theta_1},\ldots,e^{i\theta_d})$. For every $j=1,2,\ldots,d$, let $T_j\in C^\infty(X,TX)$ be the real vector field on $X$ given by
\[(T_ju)(x)=\frac{\pr}{\pr\theta_j}u((1,\ldots,1,e^{i\theta_j},1,\ldots,1)\circ x)|_{\theta_j=0},\ \ u\in C^\infty(X).\]
We assume throughout that

\begin{ass}\label{a-gue180529}
\begin{equation}\label{e-gue180528}
[T_j, C^\infty(X,T^{1,0}X)]\subset C^\infty(X, T^{1,0}X),\ \ j=1,2,\ldots,d,
\end{equation}
and there are $\mu_j\in\mathbb{R}$, $j=1,\ldots,d$, such that $\set{\mu_1,\ldots,\mu_d}$ are linear independent over $\mathbb{Q}$ and
\begin{equation}\label{e-gue180528I}
\mbox{$\Complex TX=T^{1,0}X\oplus T^{0,1}X\oplus\Complex(\mu_1T_1+\cdots+\mu_dT_d)$ on $X$}.
\end{equation}
\end{ass}

Let $T=\mu_1T_1+\cdots+\mu_dT_d$, where $\mu_j\in\mathbb{R}$, $j=1,\ldots,d$, are as in \eqref{e-gue180528I}. Let $\omega_0\in C^\infty(X,T^*X)$ be the
global non-vanishing real $1$-form on $X$ such that $\langle\,\omega_0\,,\,T\,\rangle=-1$ on $X$ and $\langle\,\omega_0\,,\,u\,\rangle=0$, for every $u\in T^{1,0}X\oplus T^{0,1}X$. The Levi form of $X$ at $x\in X$ is the Hermitian quadratic form on $T^{1,0}_xX$ given by
\begin{equation}\label{e-gue180528aI}
\mathcal{L}_x(U,\ol V)=-\frac{1}{2i}\langle\,d\omega_0(x)\,,\,U\wedge\ol V\,\rangle,\ \ \forall U, V\in T^{1,0}_xX.
\end{equation}
We assume that the Levi form is positive on $X$.
From now on, we fix a torus invariant Hermitian metric $\langle\,\cdot\,|\,\cdot\,\rangle$ on $\Complex TX$ such that $T^{1,0}X\perp T^{0,1}X$,
$T\perp(T^{1,0}X\oplus T^{0,1}X)$, $\langle\,T\,|\,T\,\rangle=1$. For $x\in X$, let ${\rm det\,}\mathcal{L}_x=\lambda_1(x)\cdots\lambda_n(x)$, where $\lambda_j(x)$, $j=1,\ldots,n$, are the eigenvalues of the Levi form $\mathcal{L}_x$ with respect to $\langle\,\cdot\,|\,\cdot\,\rangle$. Let $(\,\cdot\,|\,\cdot\,)$ be the $L^2$ inner product on $C^\infty(X)$ induced by $\langle\,\cdot\,|\,\cdot\,\rangle$. Let $L^2(X)$ be the completion of $C^\infty(X)$ with respect to $(\,\cdot\,|\,\cdot\,)$ and we extend $(\,\cdot\,|\,\cdot\,)$
to $L^2(X)$ in the standard way and we write $\norm{\cdot}$ to denote the corresponding norm. For $(p_1,\ldots,p_d)\in\mathbb{Z}^d$, let
\[C^\infty_{p_1,\ldots,p_d}(X):=\set{u\in C^\infty(X);\, u((e^{i\theta_1},\ldots,e^{i\theta_d})\circ x)=e^{ip_1\theta_1+\cdots+ip_d\theta_d}u(x),\ \ \forall x\in X}\]
and let $L^2_{p_1,\ldots,p_d}(X)$ be the completion of $C^\infty_{p_1,\ldots,p_d}(X)$ with respect to $(\,\cdot\,|\,\cdot\,)$. Fix $(p_1,\ldots,p_d)\in\mathbb{Z}^d$, let
\begin{equation}\label{e-gue180605a}
H^0_{b,p_1,\ldots,p_d}(X)=\set{u\in L^2_{p_1,\ldots,p_d}(X);\, \ddbar_bu=0},
\end{equation}
where $\ddbar_b: C^\infty(X)\To\Omega^{0,1}(X)$ is the tangential Cauchy-Riemann operator (see \eqref{e-gue180604}). Let
\[S_{p_1,\ldots,p_d}: L^2(X)\To H^0_{b,p_1,\ldots,p_d}(X)\]
be the orthogonal projection with respect to $(\,\cdot\,|\,\cdot\,)$ and let $S_{p_1,\ldots,p_d}(x,y)\in D'(X\times X)$ be the distribution kernel of $S_{p_1,\ldots,p_d}$.
From the transversal condition \eqref{e-gue180528I}, it is easy to see that $H^0_{b,p_1,\ldots,p_d}(X)$ is a finite dimensional subspace of $C^\infty_{p_1,\ldots,p_d}(X)$ and
$S_{p_1,\ldots,p_d}(x,y)\in C^\infty(X\times X)$. Let $\set{f_1,\ldots,f_d}$ be an orthonormal basis of $H^0_{b,p_1,\ldots,p_d}(X)$. Put
\begin{equation}\label{e-gue180528II}
S_{p_1,\ldots,p_d}(x):=S_{p_1,\ldots,p_d}(x,x)=\sum^d_{j=1}\abs{f_j(x)}^2,\ \ \forall x\in X.
\end{equation}

For $\tau(t)\in C^\infty_0(]0,+\infty[)$ consider the projection
\begin{equation}\label{e-gue180528III}
S_{k,\tau}: u\in L^2(X)\To\sum_{(p_1,\ldots,p_d)\in\mathbb Z^d}\tau\left(\frac{\mu_1p_1+\cdots+\mu_dp_d}{k}\right)(S_{p_1,\ldots,p_d}u)(x),
\end{equation}
where $\mu_j\in\mathbb{R}$, $j=1,2,\ldots,d$, are as in \eqref{e-gue180528I}. Let $S_{k,\tau}(x,y)\in D'(X\times X)$ be the distribution kernel of $S_{k,\tau}$.
For every $\lambda\in\Real_+$, it was shown in Lemma 4.6 in~\cite{HHL17} that the space $\oplus_{(p_1,\ldots,p_d)\in\mathbb Z^d,\abs{\mu_1p_1+\cdots+\mu_dp_d}\leq\lambda}H^0_{b,p_1,\ldots,p_d}(X)$ is finite dimensional. Hence, $S_{k,\tau}(x,y)\in C^\infty(X\times X)$. Then,
\begin{equation}\label{e-gue180528a}
S_{k,\tau}(x):=S_{k,\tau}(x,x)=\sum_{(p_1,\ldots,p_d)\in\mathbb Z^d}\tau\left(\frac{\mu_1p_1+\cdots+\mu_dp_d}{k}\right)S_{p_1,\ldots,p_d}(x).
\end{equation}
The first main result of this work is the following

\begin{theorem}\label{t-gue180528}
With the notations and assumptions used above, there are $a_j(x)\in C^\infty(X)$, $j=0,1,2,\ldots$, with
\begin{equation}\label{e-gue180529m}
a_0(x)=(2\pi)^{-n-1}\abs{{\rm det\,}\mathcal{L}_x}\int (2t)^n\abs{\tau(t)}^2dt
\end{equation}
such that for every $N\in\mathbb N$ and $\ell\in\mathbb N$, there is a constant $C_{N,\ell}>0$ independent of $k$ such that
\begin{equation}\label{e-gue180529mI}
\norm{S_{k,\tau}(x)-\sum^N_{j=0}a_j(x)k^{n+1-j}}_{C^\ell(X)}\leq C_{N,\ell}k^{n-N}.
\end{equation}
\end{theorem}

Theorem~\ref{t-gue180528} tells us that $S_{k,\tau}(x)$ admits an asymptotic expansion in $k$:
\[S_{k,\tau}(x)\sim k^{n+1}a_0(x)+k^na_1(x)+\cdots.\]

Without the cut-off function $\tau$, it is difficult to get a full asymptotic expansion but we can get leading term of the Szeg\H{o} kernel. Consider the projection
\begin{equation}\label{e-gue180529mII}
S_{k}: u\in L^2(X)\To\sum_{(p_1,\ldots,p_d)\in\mathbb Z^d, 0\leq\mu_1p_1+\cdots+\mu_dp_d\leq k}(S_{p_1,\ldots,p_d}u)(x),
\end{equation}
where $\mu_j\in\mathbb{R}$, $j=1,2,\ldots,d$, are as in \eqref{e-gue180528I}. Let $S_{k}(x,y)\in C^\infty(X\times X)$ be the distribution kernel of $S_{k}$. Then,
\begin{equation}\label{e-gue180529mIII}
S_{k}(x):=S_{k}(x,x)=\sum_{(p_1,\ldots,p_d)\in\mathbb Z^d, 0\leq\mu_1p_1+\cdots+\mu_dp_d\leq k}S_{p_1,\ldots,p_d}(x).
\end{equation}

The second main result of this work is the following

\begin{theorem}\label{t-gue180528I}
With the notations and assumptions used above, we have
\begin{equation}\label{e-gue180529mz}
\lim_{k\To+\infty}k^{-n-1}S_k(x)=\frac{1}{2}\pi^{-n-1}\frac{1}{n+1}\abs{{\rm det\,}\mathcal{L}_x}\ \ \mbox{at every $x\in X$}.
\end{equation}
\end{theorem}

\section{Preliminaries}\label{s:prelim}

\subsection{Some standard notations}\label{s-gue150508b}
We use the following notations: $\mathbb N=\set{1,2,\ldots}$,
$\mathbb N_0=\mathbb N\cup\set{0}$, $\Real$
is the set of real numbers,
\[\Real_+:=\set{x\in\Real;\, x>0},\ \ \ol\Real_+:=\set{x\in\Real;\, x\geq0}.\]
For a multiindex $\alpha=(\alpha_1,\ldots,\alpha_m)\in\mathbb N_0^m$
we set $\abs{\alpha}=\alpha_1+\cdots+\alpha_m$. For $x=(x_1,\ldots,x_m)\in\Real^m$ we write
\[
\begin{split}
&x^\alpha=x_1^{\alpha_1}\ldots x^{\alpha_m}_m,\quad
 \pr_{x_j}=\frac{\pr}{\pr x_j}\,,\quad
\pr^\alpha_x=\pr^{\alpha_1}_{x_1}\ldots\pr^{\alpha_m}_{x_m}=\frac{\pr^{\abs{\alpha}}}{\pr x^\alpha}\,,\\
&D_{x_j}=\frac{1}{i}\pr_{x_j}\,,\quad D^\alpha_x=D^{\alpha_1}_{x_1}\ldots D^{\alpha_m}_{x_m}\,,
\quad D_x=\frac{1}{i}\pr_x\,.
\end{split}
\]
Let $z=(z_1,\ldots,z_m)$, $z_j=x_{2j-1}+ix_{2j}$, $j=1,\ldots,m$, be coordinates of $\Complex^m$,
where
$x=(x_1,\ldots,x_{2m})\in\Real^{2m}$ are coordinates in $\Real^{2m}$.
Throughout the paper we also use the notation
$w=(w_1,\ldots,w_m)\in\Complex^m$, $w_j=y_{2j-1}+iy_{2j}$, $j=1,\ldots,m$, where
$y=(y_1,\ldots,y_{2m})\in\Real^{2m}$.
We write
\[
\begin{split}
&z^\alpha=z_1^{\alpha_1}\ldots z^{\alpha_m}_m\,,\quad\ol z^\alpha=\ol z_1^{\alpha_1}\ldots\ol z^{\alpha_m}_m\,,\\
&\pr_{z_j}=\frac{\pr}{\pr z_j}=
\frac{1}{2}\Big(\frac{\pr}{\pr x_{2j-1}}-i\frac{\pr}{\pr x_{2j}}\Big)\,,\quad\pr_{\ol z_j}=
\frac{\pr}{\pr\ol z_j}=\frac{1}{2}\Big(\frac{\pr}{\pr x_{2j-1}}+i\frac{\pr}{\pr x_{2j}}\Big),\\
&\pr^\alpha_z=\pr^{\alpha_1}_{z_1}\ldots\pr^{\alpha_m}_{z_m}=\frac{\pr^{\abs{\alpha}}}{\pr z^\alpha}\,,\quad
\pr^\alpha_{\ol z}=\pr^{\alpha_1}_{\ol z_1}\ldots\pr^{\alpha_m}_{\ol z_m}=
\frac{\pr^{\abs{\alpha}}}{\pr\ol z^\alpha}\,.
\end{split}
\]

Let $X$ be a $C^\infty$ orientable paracompact manifold.
We let $TX$ and $T^*X$ denote the tangent bundle of $X$ and the cotangent bundle of $X$ respectively.
The complexified tangent bundle of $X$ and the complexified cotangent bundle of $X$
will be denoted by $\Complex TX$ and $\Complex T^*X$ respectively. We write $\langle\,\cdot\,,\cdot\,\rangle$
to denote the pointwise duality between $TX$ and $T^*X$.
We extend $\langle\,\cdot\,,\cdot\,\rangle$ bilinearly to $\Complex TX\times\Complex T^*X$.

Let $E$ be a $C^\infty$ vector bundle over $X$. The fiber of $E$ at $x\in X$ will be denoted by $E_x$.
Let $F$ be another vector bundle over $X$. We write
$F\boxtimes E^*$ to denote the vector bundle over $X\times X$ with fiber over $(x, y)\in X\times X$
consisting of the linear maps from $E_y$ to $F_x$.

Let $Y\subset X$ be an open set. The spaces of
smooth sections of $E$ over $Y$ and distribution sections of $E$ over $Y$ will be denoted by $C^\infty(Y, E)$ and $\mathscr D'(Y, E)$ respectively.
Let $\mathscr E'(Y, E)$ be the subspace of $\mathscr D'(Y, E)$ whose elements have compact support in $Y$.
For $m\in\Real$, let $H^m(Y, E)$ denote the Sobolev space
of order $m$ of sections of $E$ over $Y$. Put
\begin{gather*}
H^m_{\rm loc\,}(Y, E)=\big\{u\in\mathscr D'(Y, E);\, \varphi u\in H^m(Y, E),
      \,\forall\varphi\in C^\infty_0(Y)\big\}\,,\\
      H^m_{\rm comp\,}(Y, E)=H^m_{\rm loc}(Y, E)\cap\mathscr E'(Y, E)\,.
\end{gather*}

We recall the definition of the semi-classical symbol spaces:

\begin{definition} \label{d-gue140826}
Let $W$ be an open set in $\Real^N$. Let
%\begin{gather*}
%S(1;W):=\Big\{a\in C^\infty(W)\,|\, \forall\alpha\in\mathbb N^N_0:
%\sup_{x\in W}\abs{\pr^\alpha a(x)}<\infty\Big\},\\
\[\begin{split}
&S^0_{{\rm loc\,}}(1;W)\\
&:=\Big\{(a(\cdot,k))_{k\in\Real};\, \forall\alpha\in\mathbb N^N_0,
\forall \chi\in C^\infty_0(W)\,:\:\sup_{k\in\Real, k\geq1}\sup_{x\in W}\abs{\pr^\alpha(\chi a(x,k))}<\infty\Big\}\,.\end{split}\]
For $m\in\Real$ let
\[
S^m_{{\rm loc}}(1):=S^m_{{\rm loc}}(1;W)=\Big\{(a(\cdot,k))_{k\in\Real};\,(k^{-m}a(\cdot,k))\in S^0_{{\rm loc\,}}(1;W)\Big\}\,.
\]
Hence $a(\cdot,k)\in S^m_{{\rm loc}}(1;W)$ if for every $\alpha\in\mathbb N^N_0$ and $\chi\in C^\infty_0(W)$, there
exists $C_\alpha>0$ independent of $k$, such that $\abs{\pr^\alpha (\chi a(\cdot,k))}\leq C_\alpha k^{m}$ on $W$, for every $k\in\mathbb R$.

Consider a sequence $a_j\in S^{m_j}_{{\rm loc\,}}(1)$, $j\in\N_0$, where $m_j\searrow-\infty$,
and let $a\in S^{m_0}_{{\rm loc\,}}(1)$. We say that
\[
a(\cdot,k)\sim
\sum\limits^\infty_{j=0}a_j(\cdot,k)\:\:\text{in $S^{m_0}_{{\rm loc\,}}(1)$},
\]
if for every
$\ell\in\N_0$ we have $a-\sum^{\ell}_{j=0}a_j\in S^{m_{\ell+1}}_{{\rm loc\,}}(1)$.
For a given sequence $a_j$ as above, we can always find such an asymptotic sum
$a$, which is unique up to an element in
$S^{-\infty}_{{\rm loc\,}}(1)=S^{-\infty}_{{\rm loc\,}}(1;W):=\cap _mS^m_{{\rm loc\,}}(1)$.

We say that $a(\cdot,k)\in S^{m}_{{\rm loc\,}}(1)$ is a classical symbol on $W$ of order $m$ if
\begin{equation} \label{e-gue13628I}
a(\cdot,k)\sim\sum\limits^\infty_{j=0}k^{m-j}a_j\: \text{in $S^{m}_{{\rm loc\,}}(1)$},\ \ a_j(x)\in
S^0_{{\rm loc\,}}(1),\ j=0,1\ldots.
\end{equation}
The set of all classical symbols on $W$ of order $m$ is denoted by
$S^{m}_{{\rm loc\,},{\rm cl\,}}(1)=S^{m}_{{\rm loc\,},{\rm cl\,}}(1;W)$.

Let $X$ be a $C^\infty$ orientable paracompact manifold. By using partition of unity, we define $S^m_{{\rm loc}}(1):=S^m_{{\rm loc}}(1;X)$, $S^{-\infty}_{{\rm loc\,}}(1)=S^{-\infty}_{{\rm loc\,}}(1;X)$, $S^{m}_{{\rm loc\,},{\rm cl\,}}(1)=S^{m}_{{\rm loc\,},{\rm cl\,}}(1;X)$ and asymptotic sum in the symbol space $S^m_{{\rm loc}}(1;X)$ in the standard way.
\end{definition}

\subsection{CR manifolds with $\Real$-action} \label{s-gue150808}

Let $(\hat X, T^{1,0}\hat X)$ be a compact CR manifold of dimension $2n+1$, $n\geq 1$, where $T^{1,0}\hat X$ is a CR structure of $\hat X$. That is $T^{1,0}\hat X$ is a subbundle of rank $n$ of the complexified tangent bundle $\mathbb{C}T\hat X$, satisfying $T^{1,0}\hat X\cap T^{0,1}\hat X=\{0\}$, where $T^{0,1}\hat X=\overline{T^{1,0}\hat X}$, and $[\mathcal V,\mathcal V]\subset\mathcal V$, where $\mathcal V=C^\infty(\hat X, T^{1,0}\hat X)$. In this section, we assume that
$\hat X$ admits a $\Real$-action $\eta$, $\eta\in\Real$: $\eta: \hat X\to\hat X$, $x\mapsto\eta\circ x$. Let $T\in C^\infty(\hat X, T\hat X)$ be the global real vector field induced by the $\Real$-action given by
\begin{equation}\label{e-gue150808}
(Tu)(x)=\frac{\partial}{\partial \eta}\left(u(\eta\circ x)\right)|_{\eta=0},\ \ u\in C^\infty(\hat X).
\end{equation}

\begin{definition}
We say that the $\Real$-action $\eta$ is CR if
\[[T, C^\infty(\hat X, T^{1,0}\hat X)]\subset C^\infty(\hat X, T^{1,0}\hat X)\]
and the $\Real$-action is transversal if for each $x\in\hat X$,
$\Complex T(x)\oplus T_x^{1,0}\hat X\oplus T_x^{0,1}\hat X=\mathbb CT_x\hat X$.
\end{definition}

Assume that $(\hat X, T^{1,0}\hat X)$ is a compact CR manifold of dimension $2n+1$, $n\geq 1$, with a transversal CR $\Real$-action $\eta$ and we let $T$ be the global vector field induced by the $\Real$-action. Let $\omega_0\in C^\infty(\hat X,T^*\hat X)$ be the global real one form determined by
\begin{equation}\label{e-gue150808I}
\begin{split}
&\langle\,\omega_0\,,\,u\,\rangle=0,\ \ \forall u\in T^{1,0}\hat X\oplus T^{0,1}\hat X,\\
&\langle\,\omega_0\,,\,T\,\rangle=-1.
\end{split}
\end{equation}
As \eqref{e-gue180528aI}, we have

\begin{definition}\label{d-gue150808}
For $p\in\hat X$, the Levi form $\mathcal L_p$ is the Hermitian quadratic form on $T^{1,0}_p\hat X$ given by
$\mathcal{L}_p(U,\ol V)=-\frac{1}{2i}\langle\,d\omega_0(p)\,,\,U\wedge\ol V\,\rangle$, $U, V\in T^{1,0}_p\hat X$.
\end{definition}

Denote by $T^{*1,0}\hat X$ and $T^{*0,1}\hat X$ the dual bundles of
$T^{1,0}\hat X$ and $T^{0,1}\hat X$ respectively. Define the vector bundle of $(0,q)$ forms by
$T^{*0,q}\hat X =\Lambda^q(T^{*0,1}\hat X)$.
Let $D\subset\hat X$ be an open set. Let $\Omega^{0,q}(D)$
denote the space of smooth sections of $T^{*0,q}\hat X$ over $D$ and let $\Omega_0^{0,q}(D)$
be the subspace of $\Omega^{0,q}(D)$ whose elements have compact support in $D$. Similarly, if $E$ is a vector bundle over $D$, then we let $\Omega^{0,q}(D, E)$ denote the space of smooth sections of $T^{*0,q}\hat X\otimes E$ over $D$ and let $\Omega_0^{0,q}(D, E)$ be the subspace of $\Omega^{0,q}(D, E)$ whose elements have compact support in $D$. In this section, we assume that
\begin{ass}\label{a-gue180605}
$\hat X$ admits a $\Real$-invariant Hermitian metric $\langle\,\cdot\,|\,\cdot\,\rangle$ on $\Complex T\hat X$ such that $T^{1,0}\hat X\perp T^{0,1}\hat X$,
$T\perp(T^{1,0}\hat X\oplus T^{0,1}\hat X)$, $\langle\,T\,|\,T\,\rangle=1$.
\end{ass}
The $\Real$-invariant Hermitian metric $\langle\,\cdot\,|\,\cdot\,\rangle$ induces by duality a $\Real$-invariant Hermitian metric $\langle\,\cdot\,|\,\cdot\,\rangle$ on $\oplus^{2n+1}_{j=1}\Lambda^j\Bigr(\Complex T^*\hat X\Bigr)$. Let
\[\tau^{0,1}: CT^*\hat X\To T^{*0,1}\hat X\]
be the orthogonal projection with respect to $\langle\,\cdot\,|\,\cdot\,\rangle$. The tangential Cauchy Riemann operator is given by
\begin{equation}\label{e-gue180604}
\ddbar_b:=\tau^{0,1}\circ d: C^\infty(\hat X)\To\Omega^{0,1}(\hat X).
\end{equation}

In the rest of this section, we will review  the Szeg\H{o} kernel asymptotic expansion established in~\cite{HHL17}. We need to introduce more definitions and notations.

\begin{definition}\label{d-gue50508d}
Let $D\subset U$ be an open set. We say that a function $u\in C^\infty(D)$ is rigid if $Tu=0$. We say that a function $u\in C^\infty(\hat X)$ is Cauchy-Riemann (CR for short)
(on $D$) if $\ddbar_bu=0$. We say that $u\in C^\infty(D)$ is rigid CR (on $D$) if $\ddbar_bu=0$ and $Tu=0$.
\end{definition}

\begin{definition} \label{d-gue150508dI}
Let $F$ be a complex vector bundle over $\hat X$. We say that $F$ is rigid (resp.\ CR, resp.\ rigid CR) if there exists
an open cover $(U_j)_j$ of $\hat X$ and trivializing frames $\set{f^1_j,f^2_j,\dots,f^r_j}$ on $U_j$,
such that the corresponding transition matrices are rigid (resp.\ CR, resp.\ rigid CR).
\end{definition}

Let $F$ be a rigid (CR) vector bundle over $X$. In this work, we fix open cover $(U_j)^N_{j=1}$ of $X$ and  a family $\set{f^1_j,f^2_j,\dots,f^r_j}^N_{j=1}$ of trivializing frames $\set{f^1_j,f^2_j,\dots,f^r_j}$ on each $U_j$ such that the entries of the transition matrices between different frames $\set{f^1_j,f^2_j,\dots,f^r_j}$ are rigid (CR). For any local trivializing frames $\set{f^1,\ldots,f^r}$ of $F$ on an open set $D$, we say that $\set{f^1,\ldots,f^r}$ is a rigid (CR) frame if the entries of the transition matrices between $\set{f^1,\ldots,f^r}$ and $\set{f^1_j,\ldots,f^r_j}$ are rigid (CR), for every $j$, and we call $D$ local rigid (CR) trivialization. By using the fix trivializing frames $\set{f^1_j,\ldots,f^r_j}^N_{j=1}$, we define the operator $T$ on $C^\infty(\hat X,F)$ in the standard way.

\begin{definition}\label{d-gue150514f}
Let $F$ be a rigid vector bundle over $\hat X$. Let $\langle\,\cdot\,|\,\cdot\,\rangle_F$ be a Hermitian metric on $F$. We say that $\langle\,\cdot\,|\,\cdot\,\rangle_F$ is a rigid Hermitian metric if for every local rigid frame $f_1,\ldots, f_r$ of $F$, we have $T\langle\,f_j\,|\,f_k\,\rangle_F=0$, for every $j,k=1,2,\ldots,r$.
\end{definition}

We notice that Definition~\ref{d-gue150514f} above depends on the fix open cover $(U_j)^N_{j=1}$ of $\hat X$ and  a family $\set{f^1_j,f^2_j,\dots,f^r_j}^N_{j=1}$ of trivializing frames $\set{f^1_j,f^2_j,\dots,f^r_j}$ on each $U_j$ such that the entries of the transition matrices between different frames $\set{f^1_j,f^2_j,\dots,f^r_j}$ are rigid.

In this section, let $L$ be a rigid CR line bundle over $\hat X$.
We fix an open covering $(U_j)^N_{j=1}$ and a family $(s_j)^N_{j=1}$ of trivializing frames $s_j$ on each $U_j$ such that the entries of the transition functions between different frames $s_j$ are rigid CR.
Let $L^k$ be the $k$-th tensor power of $L$.
Then $(s_j^{\otimes k})^N_{j=1}$ is a family of trivializing frames $s^{\otimes k}_j$ on each $U_j$. For any local trivializing frame $f$ of $L^k$ on an open set $D$, we say that $f$ is a rigid CR frame if the transition functions between $f$ and $s^{\otimes k}_j$ are rigid CR, for every $j$, and we call $D$ local rigid CR trivialization.

Since $L^k$ is CR, we can consider the tangential Cauchy-Riemann operator as an operator acting on the sections of $L^k$:
\[\overline\partial_b: C^\infty(\hat X, L^k)\rightarrow\Omega^{0,1}(\hat X, L^k).\]
Since $L^k$ is rigid, by using the fix trivializing frames $\set{s^{\otimes k}_j}^N_{j=1}$, we define $Tu$  for every $u\in C^\infty(\hat X, L^k)$ in the standard way.
Let $h^L$ be a Hermitian metric on $L$. The local weight of $h^L$
with respect to a local rigid CR trivializing section $s$ of $L^L$ over an open subset $D\subset X$
is the function $\Phi\in C^\infty(D, \mathbb R)$ for which
\begin{equation}\label{e-gue150808g}
|s(x)|^2_{h^L}=e^{-2\Phi(x)}, x\in D.
\end{equation}
We denote by $\Phi_j$ the weight of $h^L$ with respect to $s_j$.

\begin{definition}\label{d-gue150808g}
The curvature of $(L,h^L)$ is the the Hermitian quadratic form $R^L=R^{(L,h^L)}$ on $T^{1,0}\hat X$
defined by
\begin{equation}\label{e-gue150808w}
R_p^L(U, V)=\,\big\langle d(\overline\partial_b\Phi_j-\partial_b\Phi_j)(p),
U\wedge\overline V\,\big\rangle,\:\: U, V\in T_p^{1,0}\hat X,\:\: p\in U_j.
\end{equation}
\end{definition}

Due to~\cite[Proposition 4.2]{HM09}, $R^L$ is a well-defined global Hermitian form,
since the transition functions between different frames $s_j$ are annihilated by $T$.

From now on, we fix a rigid Hermitian metric $h^L$ on $L$. The Hermitian metric on $L^k$ induced by $h^L$ is denoted by $h^{L^k}$.
We denote by $dv_{\hat X}$ the volume form induced by $\langle\,\cdot\,|\,\cdot\,\rangle$.
Let $(\,\cdot\,|\,\cdot\,)_k$ be the $L^2$ inner product on $C^\infty(\hat X,L^k)$ induced by $h^{L^k}$ and
$dv_{\hat X}$. Let $L^2(\hat X,L^k)$ be the completion of $C^\infty(\hat X,L^k)$ with respect
to $(\,\cdot\,|\,\cdot\,)_k$. We extend $(\,\cdot\,|\,\cdot\,)_k$ to $L^2(\hat X,L^k)$ in the standard way and we write $\norm{\cdot}_k$ to denote the corresponding norm.
Consider the operator
\[-iT: C^\infty(\hat X,L^k)\To C^\infty(\hat X,L^k)\]
and we extend $-iT$ to the $L^2$ space by
\[\begin{split}
&-iT: {\rm Dom\,}(-iT)\subset L^2(\hat X,L^k)\To L^2(\hat X,L^k),\\
&{\rm Dom\,}(-iT)=\set{u\in L^2(\hat X,L^k);\, -iTu\in L^2(\hat X,L^k)}.
\end{split}\]
It is easy to see that $-iT$ is self-adjoint with respect to $(\,\cdot\,|\,\cdot\,)_k$. Let  ${\rm Spec\,}(-iT)$ denote the spectrum of $-iT$. We make the following assumptions

\begin{ass}\label{a-gue180604}
${\rm Spec\,}(-iT)$  is countable
and every element in ${\rm Spec\,}(-iT)$ is an eigenvalue of $-iT$.
\end{ass}

\begin{ass}\label{a-gue180604I}
There is a  non-empty open interval $I\subset\mathbb{R}$ such that
\begin{equation}\label{e-gue180604a}
\mbox{$R^L_x+2s\mathcal{L}_x$ is positive definite at every point $x\in\hat X$, for every $s\in I$}.
\end{equation}
\end{ass}

Fix $\alpha\in {\rm Spec\,}(-iT)$. Put
$C^\infty_{\alpha}(\hat X,L^k):=\set{u\in C^\infty(X,L^k);\, -iTu=\alpha u}$ and let
\begin{equation}\label{e-gue150806II}
H^0_{b,\alpha}(\hat X,L^k):=\set{u\in C^\infty_{\alpha}(\hat X,L^k);\, \ddbar_bu=0}.
\end{equation}
It is easy to see that
\begin{equation}\label{e-gue150806III}
{\rm dim\,}H^{0}_{b,\alpha}(\hat X,L^k)<\infty.
\end{equation}
Let $\set{f_1,\ldots,f_{d_k}}$ be an orthonormal basis for $H^{0}_{b,\alpha}(\hat X,L^k)$ with respect to $(\,\cdot\,|\,\cdot\,)_k$. Put
\begin{equation}\label{e-gue180604m}
\hat S_{k,\alpha}(x):=\sum^{d_k}_{j=1}\abs{f_j(x)}^2_{h^{L^k}}\in C^\infty(\hat X).
\end{equation}

Fix a function
\begin{equation}\label{e-gue160105}
\tau(t)\in C^\infty_0(I),
\end{equation}
where $I$ is the open interval in Assumption~\ref{a-gue180604I}. Define the weighted Fourier-Szeg\H{o} kernel function by:
\begin{equation}\label{e-gue180605}
\hat S_{k,\tau}(x):=\sum_{\alpha\in{\rm Spec\,}(-iT)}\tau\left(\frac{\alpha}{k}\right)\hat S_{k,\alpha}(x)\in C^\infty(\hat X).
\end{equation}
By Lemma 4.6 in~\cite{HHL17}, the sum in \eqref{e-gue180605} is a finite sum, hence $\hat S_{k,\tau}(x)$ is well-defined as a smooth function on $\hat X$.

We have the following:

\begin{theorem}\label{t-gue180605}
With the assumptions and notations above and recall that we work with Assumption~\ref{a-gue180605}, Assumption~\ref{a-gue180604} and Assumption~\ref{a-gue180604I}. We have
\begin{equation}\label{e-gue180605m}
\hat S_{k,\tau}(x)\sim\sum^\infty_{j=0}a_j(x)k^{n+1-j}\text{ in }S^{n+1}_{{\rm loc\,}}
(1;\hat X),
\end{equation}
where $a_j(x)\in C^\infty(\hat X)$, $j=0,1,2,\ldots$,
and
\begin{equation}\label{e-gue150807a}
a_0(x)=(2\pi)^{-n-1}\int\abs{\det\bigr(R^L_x+2t\mathcal{L}_x\bigr)}\abs{\tau(t)}^2dt,\ \ \forall x\in\hat X.
\end{equation}
\end{theorem}

For the proof of Theorem~\ref{t-gue180605}, we refer the reader to Theorem 1.1 in~\cite{HHL17}. It should be mentioned that in~\cite{HHL17}, we assume that $R^L$ is positive definite on $\hat X$ but if we go through the proof of Theorem 1.1 in~\cite{HHL17}, we only need the positivity condition as in Assumption~\ref{a-gue180604I}.

\section{Asymptotic upper bound for the Szeg\H{o} kernel}\label{s-gue180615rI}

In this section, we will first establish an asymptotic upper bound for the Szeg\H{o} kernel.
We  consider the following general setting.
Let $(\hat X, T^{1,0}\hat X)$ be a compact CR manifold of dimension $2n+1$, $n\geq 1$, with a transversal CR $\Real$-action $\eta$ and we let $T$ be the global vector field induced by the $\Real$-action. Let $\omega_0\in C^\infty(\hat X,T^*\hat X)$ be the global real one form determined by \eqref{e-gue150808I}. Let $(L,h^L)$ be a rigid CR line bundle over $\hat X$ and let $(L^k,h^{L^k})$ be the $k$-th power of $(L,h^L)$, where $h^L$ denotes the given rigid Hermitian fiber metric on $L$. We will
use the same notations as in Section~\ref{s-gue150808} and assume that Assumption~\ref{a-gue180605} and Assumption~\ref{a-gue180604} hold. Note that we don't assume that Assumption~\ref{a-gue180604I} holds. For $k\in\mathbb N$, put
\begin{equation}\label{e-gue180612m}
\hat S_k(x):=\sum_{\alpha\in{\rm Spec\,}(-iT), 0\leq\alpha\leq k}\hat S_{k,\alpha}(x)\in C^\infty(\hat X),
\end{equation}
where $\hat S_{k,\alpha}(x)$ is as in \eqref{e-gue180604m}. It was shown in Lemma 4.6 in~\cite{HHL17} that for every $k\in\mathbb N$, there are only finitely many $\alpha\in{\rm Spec\,}(-iT)$ with $0\leq\alpha\leq k$ such that $\hat S_{k,\alpha}(x)$ is not equal to the zero function on $\hat X$ and hence the right hand side of \eqref{e-gue180612m} is a finite sum. In this section, we will give an upper bound of the function $\limsup_{k\To+\infty}k^{-(n+1)}\hat S_k(x)$. Consider the space
\[H^0_{b,0\leq\alpha\leq k}(\hat X, L^k):=\oplus_{\alpha\in{\rm Spec\,}(-iT), 0\leq\alpha\leq k}H^0_{b,\alpha}(\hat X,L^k),\]
where $H^0_{b,\alpha}(\hat X,L^k)$ is given by \eqref{e-gue150806II}.
The following is well-known ( see Lemma 2.1 in~\cite{HM09})

\begin{lemma}\label{l-gue180615m}
For every $x\in\hat X$, we have
\begin{equation}\label{e-gue180615p}
\hat S_k(x)=\sup\limits_{u\in H^0_{b, 0\leq\alpha\leq k}(\hat X,L^k), \|u\|_{k}=1}|u(x)|_{h^{L^k}}^2.
\end{equation}
\end{lemma}

We first recall the CR scaling technique developed in~\cite{HM09} and~\cite{HL15}.

Fix $p\in\hat X$. There exist local coordinates
\[(x_1,\cdots,x_{2n+1})=(z,\theta)=(z_1,\cdots,z_{n},\theta),\ \ z_j=x_{2j-1}+ix_{2j},\ \ j=1,\cdots,n, x_{2n+1}=\theta,\]
and local rigid CR trivialization section $s$, $\abs{s}^2_{h^L}=e^{-2\Phi}$, defined in some small neighborhood $D$ centered at $p$ such that on $D$ one has
\begin{equation}\label{e-can}
\begin{split}
&T=\frac{\partial}{\partial\theta},\\
&U_j=\frac{\partial}{\partial z_j}+i\lambda_j\ol z_j\frac{\partial}{\partial\theta}+O(\abs{(z,\theta)}^2),  j=1,\cdots,n,\\
%&\varphi(z)=\sum^n_{j=1}\lambda_j\abs{z_j}^2+O(\abs{z}^3),\\
%&\mbox{$\frac{\partial}{\partial z_1}(p),\ldots,\frac{\pr}{\pr z_n}(p)$ is an orthnormal basis for $T^{1,0}_p\hat X$},\\
&\Phi=\sum\limits_{j,t=1}^{n-1}\mu_{j,t}\overline z_jz_t+O(|z|^3),
\end{split}
\end{equation}
where $U_j(x), j=1,\cdots, n-1$, is an orthnormal basis of $T_x^{1,0}\hat X$, for each $x\in D$ (see the discussion in the beginning of Section 2.1 in~\cite{HL15}). Note that $\lambda_1,\ldots,\lambda_n$ are the eigenvalues of the Levi form $\mathcal{L}_x$ with respect to the given Hermitian metric $\langle\,\cdot\,|\,\cdot\,\rangle$. Until further notice,  we work with the local coordinates $z=(z,\theta)$ and we identify $D$ with some open set in $\mathbb R^{2n+1}$. Let $(\,\cdot\,|\,\cdot\,)_{k\Phi}$ be the weighted inner product on the space $\Omega^{0,q}_0(D)$ defined as follows:
\begin{equation}
(\,f\,|\,g\,)_{k\Phi}=\int_D\langle\,f\,|\,g\,\rangle e^{-2k\Phi(z)}dv_{\hat X}(x),
\end{equation}
where $f, g\in\Omega_0^{0,q}(D)$ and $dv_{\hat X}$ is the volume form on $\hat X$ induced by $\langle\,\cdot\,|\,\cdot\,\rangle$.
We denote by $L^2_{(0,q)}(D,k\Phi)$  the completion of $\Omega_0^{0,q}(D)$ with respect to $(\,\cdot\,|\,\cdot\,)_{k\Phi}$ and we write $L^2(D,k\Phi):=L^2_{(0,0)}(D,k\Phi)$. For $r>0$, let
$D_r=\{(z,\theta)\in\mathbb R^{2n+1};\, |z_j|<r, j=1,\ldots,n, |\theta|<r\}$. Let $F_k$ be the scaling map $F_k(z,\theta)=(\frac{z}{\sqrt k}, \frac{\theta}{k})$. From now on, we
assume $k$ is sufficiently large such that $F_k(D_{\log k})\Subset D$. Let
$\set{e_1,\ldots,e_n}$
be an orthonormal basis for $T^{*0,1}\hat X$ which is dual to $U_1,\ldots,U_n$.
We define the scaled bundle $F_k^\ast T^{\ast0, q}\hat X$ on $D_{\log k}$ to be the bundle whose fiber at
$(z,\theta)\in D_{\log k}$ is
\begin{equation}
F_k^\ast T^{\ast0, q}\hat X|_{(z,\theta)}=\left\{\sum\nolimits_{|J|=q}a_Je^J\left(\frac{z}{\sqrt k}, \frac{\theta}{k}\right): a_J\in\mathbb C, |J|=q, J~\text{strictly increasing}\right\},
\end{equation}
where for $J=(j_1,\ldots,j_q)$, $e^J:=e_{j_1}\wedge\cdots\wedge e_{j_q}$.
We take the Hermitian metric $\langle\,\cdot\,|\,\cdot\,\rangle_{F_k^\ast}$ on $F_k^\ast T^{\ast 0, q}\hat X$ so that at each point $(z,\theta)\in D_{\log k}$,
\begin{equation}
\left\{e^J\left(\frac{z}{\sqrt k}, \frac{\theta}{k}\right);\, |J|=q, J~\text{strictly increasing}\right\}
\end{equation}
is an orthonormal frame for $F_k^\ast T^{\ast0, q}\hat X$ on $D_{\log k}$. Let $F^\ast_k\Omega^{0,q}(D_r)$ denote the space of smooth sections of $F^\ast_kT^{*0,q}\hat X$ over $D_r$ and let $F^\ast_k\Omega^{0,q}_0(D_r)$ be the subspace of $F^\ast_k\Omega^{0,q}(D_r)$ whose elements have compact support in $D_r$.
Given $f\in\Omega^{0,q}(D)$ 
we write $f=\sum\nolimits_{|J|=q}^\prime f_Je^J$, where the prime means the multiindex in the summation is strictly increasing. We define the scaled form $F_k^\ast f\in F_k^\ast\Omega^{0,q}(D_{\log k})$ by
\begin{equation}
F_k^\ast f=\sum\nolimits_{|J|=q}^\prime f_J\left(\frac{z}{\sqrt k}, \frac{\theta}{k}\right)e^J\left(\frac{z}{\sqrt k}, \frac{\theta}{k}\right).
\end{equation}
For brevity, we denote $F_k^\ast f$ by $f(\frac{z}{\sqrt k}, \frac{\theta}{k})$.
Let $P$ be a partial differential operator of order one on $F_k(D_{\log k})$ with $C^\infty$ coefficients. We write $P=a(z,\theta)\frac{\partial}{\partial\theta}+\sum\limits_{j=1}^{2n}a_j(z, \theta)\frac{\partial}{\partial x_j}.$ The scaled partial differential operator $P_{(k)}$ on $D_{\log k}$
is given by
\begin{equation}
P_{(k)}=\sqrt{k}F_k^\ast a\frac{\partial}{\partial\theta}+\sum_{j=1}^{2n}F_k^\ast a_j\frac{\partial}{\partial x_j}.
\end{equation}
Let $f\in C^\infty(F_k(D_{\log k}))$. We can check that
\begin{equation}\label{e-gue180625}
P_{(k)}(F_k^\ast f)=\frac{1}{\sqrt k}F_k^\ast(Pf).
\end{equation}
The scaled differential operator $\overline\partial_{b,(k)}: C^\infty(D_{\log k})\rightarrow F_k^\ast\Omega^{0,1}(D_{\log k})$ is given by
\begin{equation}\label{ee1}
\overline\partial_{b,(k)}=\sum_{j=1}^{n}e_j\left(\frac{z}{\sqrt k}, \frac{\theta}{k}\right)\wedge\overline U_{j,(k)}.
\end{equation}
From (\ref{ee1}) and \eqref{e-gue180625}, $\overline\partial_{b,(k)}$ satisfies that
\begin{equation}\label{l1}
\overline\partial_{b,(k)}F_k^\ast f=\frac{1}{\sqrt{k}}F_k^\ast(\overline\partial_b f).
\end{equation}
Let $(\,\cdot\,|\,\cdot\,)_{kF_k^\ast\Phi}$ be the inner product on the space $F_k^\ast\Omega^{0,q}_0(D_{\log k})$ defined  as follows:
\begin{equation}
(\,f\,|\,g\,)_{kF_k^\ast\Phi}=\int_{D_{\log k}}\langle\,f\,|\,g\,\rangle_{F_k^\ast}e^{-2kF_k^\ast\Phi}(F_k^\ast m)dv(z)d\theta,
\end{equation}
where $dv_{\hat X}=m(z)dv(z)d\theta$ on $D$, $m(z)\in C^\infty(D)$, $dv(z)=2^ndx_1\cdots dx_{2n}$.
Let
\[\overline\partial^\ast_{b, (k)}: F_k^\ast\Omega^{0,1}(D_{\log k})\rightarrow C^\infty(D_{\log k})\]
be the formal adjoint of $\overline\partial_{b, (k)}$ with respect to $(\cdot|\cdot)_{kF_k^\ast\Phi}$.
We define the scaled Kohn Laplacian $\Box_{b,(k)}: C^\infty(D_{\log k})\rightarrow C^\infty(D_{\log k})$ which is given by
\begin{equation}\label{e-gue180615}
\Box_{b,(k)}=\overline\partial_{b,(k)}^\ast\overline\partial_{b,(k)}.
\end{equation}
Let $U\subset D_{\log k}$ be an open set . Given a function $u\in C^\infty(D_{\log k})$ the Sobolev norm of $u$ of order $s$ with respect to the weight $kF_k^\ast\Phi$ is given by
\begin{equation}
\|u\|^2_{kF_k^\ast\Phi,s,U}:=\sum\limits_{\alpha\in\mathbb N_0^{2n+1}, |\alpha|\leq s}
\int_{U}|\partial^\alpha_{x,\theta}u|^2e^{-2kF_k^\ast\Phi}(F_k^\ast
m)dv(z)d\theta.
\end{equation}
We write $\norm{u}^2_{kF^\ast_k\Phi,U}:=\norm{u}^2_{kF^\ast_k\Phi,0,U}$. We have the following estimate (see Proposition 2.3 in~\cite{HL15}).

\begin{proposition}\label{p-gue180615}
For every $r>0$ with $D_{2r}\Subset D_{\log k}$ and every $s\in\mathbb N$, there exists a constant $C_{r,s}>0$ independent of $k$ and the point $p$ such that for all $u\in C^\infty(D_{\log k})$, we have
\begin{equation}\label{e-gue180615m}
\|u\|_{kF_k^\ast\Phi, s+1, D_r}^2\leq C_{r,s}\left(\|u\|^2_{kF_k^\ast\Phi, D_{2r}}+\|\Box_{b,(k)}u\|^2_{kF_k^\ast\Phi, s, D_{2r}}
+\left\|\left(\frac{\partial}{\partial\theta}\right)^{s+1}u\right\|^2_{kF_k^\ast\Phi,D_{2r}}\right).
\end{equation}
\end{proposition}

\subsection{The Heisenberg guoup $ H_n$}\label{s-gue180615r}

We pause and introduce some notations.
We identify $\mathbb R^{2n+1}$ with the Heisenberg gruop $H_{n+1}:=\mathbb C^{n}\times\mathbb R.$ We also write $(z,\theta)$ to denote the coordinates of $H_{n+1}$, $z=(z_1,\cdots, z_{n}),\theta\in\mathbb R,$ $z_j=x_{2j-1}+i x_{2j}, j=1,\cdots, n$. Take
\begin{equation}
\left\{U_{j,  H_{n+1}}=\frac{\partial}{\partial z_j}+i\lambda_j\overline z_j\frac{\partial}{\partial\theta};\, j=1,\cdots, n\right\}
\end{equation}
be the CR structure on $H_{n+1}$.
%and
%$$\left\{U_{j,  H_{n+1}}, \overline {U_{j,  H_{n+1}}}, T=\frac{\partial}{\partial\theta};\, j=1,\cdots,n\right\}$$
%are local frames for the bundles of $T^{1,0}H_{n+1}$ and $\mathbb CTH_{n+1}$. Then
%\begin{equation}
%\left\{dz_j, d\overline z_j, \omega_0=-d\theta+\sum_{j=1}^{n}(i\lambda_j\overline z_jdz_j-i\lambda_jz_jd\overline z_j);\, j=1,\cdots, n\right\}
%\end{equation}
%is the basis of $\mathbb CT^\ast H_{n+1}$ which are dual to $\{U_{j, H_{n+1}}, \overline{U_{j, H_{n+1}}}, -T\}$. Let $\langle\,\cdot\,|\,\cdot\,\rangle$ be the Hermitian metric defined on $T^{\ast0,q}H_{n+1}$ such that $\{d\overline z^J: |J|=q; J~\text{strictly~increasing}\}$ is an orthonormal frame of $T^{\ast0, q} H_{n+1}$. Let
%\begin{equation}
%\overline\partial_{b, H_n}=\sum_{j=1}^{n}d\overline z_j\wedge\overline{U_{j, H_n}}: C^\infty(H_{n+1})\rightarrow\Omega^{0,1}(H_n)
%\end{equation}
%be the Cauchy-Riemann operator defined on $H_n$.
Put $\Phi_0(z)=\sum\limits_{j,t=1}^{n}\mu_{j,t}\overline z_jz_t\in C^\infty(H_{n+1}, \mathbb R)$.
Recall that $\lambda_j$, $j=1,\ldots,n$, $\mu_{j,t}$, $j, t=1,\ldots,n$, are as in \eqref{e-can}.
Let $(\,\cdot\,|\,\cdot\,)_{\Phi_0}$ be the inner product on $C^\infty_0(H_n)$ with respect to the weight function $\Phi_0(z)$ defined as follows:
\begin{equation}
(\,f\,|\,g\,)_{\Phi_0}=\int_{H_{n+1}}\langle\,f\,|\,g\,\rangle e^{-2\Phi_0(z)}dv(z)d\theta
\end{equation}
with $dv(z)=2^{n}dx_1\cdots dx_{2n}$. We denote by $\|\cdot\|_{\Phi_0}$ the norm on $C^\infty_0(H_{n+1})$ induced by the inner product $(\,\cdot\,|\,\cdot\,)_{\Phi_0}$. Let $L^2(H_{n+1}, \Phi_0)$ be the completion of $C^\infty_0(H_n)$ with respect to the norm $\|\cdot\|_{\Phi_0}$.
%Put
%\[H^0_b(H_{n+1},\Phi_0):=\set{u\in L^2(H_{n+1}, \Phi_0);\, \ol U_{j,  H_{n+1}}u=0,\ \ \forall j=1,\ldots,n}.\]

Choose $\chi(\theta)\in C_0^\infty(\mathbb R)$
so that $\chi(\theta)=1$ when $|\theta|<1$ and
$\chi(\theta)=0$ when $|\theta|>2$
and set $\chi_j(\theta)=\chi(\frac{\theta}{j}),j\in\mathbb N$. For any $u(z, \theta)\in C^\infty(H_{n+1})$ with $\|u\|_{\Phi_0}<\infty$, set
\begin{equation}\label{e-gue180615c}
\hat u_j(z, \eta)=\int_{\mathbb R} u(z, \theta)\chi_j(\theta)e^{-i\theta\eta}d\theta\in C^\infty(H_{n+1}),\,\, j=1, 2,\ldots.
\end{equation}
From Parseval's formula, $\{\hat u_j(z, \eta)\}$ is a Cauchy sequence in $L^2(H_{n+1}, \Phi_0)$. Thus there is $\hat u(z, \eta)\in L^2(H_{n+1}, \Phi_0)$ such that $\hat u_j(z, \eta)\rightarrow \hat u(z, \eta)$ in $L^2(H_{n+1}, \Phi_0)$. We call $\hat u(z, \eta)$ the partial Fourier transform of $u(z, \theta)$ with respect to $\theta$. From Parseval's formula, we can check that
\begin{equation} \label{s4-e12-1}
\begin{split}
\int_{H_{n+1}}\!\abs{\hat u(z, \eta)}^2e^{-2\Phi_0(z)}dv(z)d\eta=
2\pi\int_{H_n}\!\abs{u(z, \theta)}^2e^{-2\Phi_0(z)}dv(z)d\theta.
\end{split}
\end{equation}
Let $s\in L^2(H_{n+1}, \Phi_0)$ be a function such that
 $\int\!\abs{s(z, \eta)}^2d\eta<\infty$ and $\int\!\abs{s(z, \eta)}d\eta<\infty$ holds for all $z\in\Complex^{n}$. Then, from Parseval's formula, we find
\begin{equation} \label{s4-e12-2}
\begin{split}
&\iint\!\hat u(z, \eta)\ol{s(z, \eta)}\,e^{-2\Phi_0(z)}d\eta dv(z)\\
&=\iint\!u(z, \theta)\ol{\int\! e^{i\theta\eta}s(z, \eta)d\eta}e^{-2\Phi_0(z)}d\theta dv(z).
\end{split}
\end{equation}
%Put
%\begin{equation}\label{e-gue180615e}
%H^0_{b,[0,1]}(H_{n+1},\Phi_0):=\{u\in H^0_b(H_{n+1},\Phi_0);\, \mbox{$\hat u(z,\eta)=0$ for almost every $\eta\not\in[0,1]$}\}.
%\end{equation}

For fixed $\eta\in\mathbb R$, put $\Phi_{\eta}(z)=\eta\sum\limits_{j=1}^{n}\lambda_j|z_j|^2+\sum\limits_{j,t=1}^{n}\mu_{j,t}\overline z_jz_t$.
Let $(\,\cdot\,|\,\cdot\, )_{\Phi_\eta}$ be the inner product on $C^\infty_0(\Complex^{n})$ defined by
\[(\,f\,|\,g\,)_{\Phi_\eta}=\int_{\Complex^{n}}\!f(z)\ol{g(z)}e^{-2\Phi_\eta(z)}dv(z)\,,\quad f, g\in C^\infty_0(\Complex^{n})\,,\]
where $dv(z)=2^{n}dx_1dx_2\cdots dx_{2n}$, and let $\norm{\cdot}_{\Phi_\eta}$ denote the corresponding norm. Let us denote by $L^2(\mathbb C^{n}, \Phi_\eta)$ the completion of $C^\infty_0(\mathbb C^n)$ with respect to the norm $\|\cdot\|_{\Phi_\eta}$.
Let
\[B_{\Phi_\eta}:L^2(\mathbb C^{n}, \Phi_\eta)\rightarrow{\rm Ker\,}\ddbar\]
be the Bergman projection with respect to $(\,\cdot\,|\,\cdot\, )_{\Phi_\eta}$ and let $B_{\Phi_{\eta}}(z,w)$ be the distribution kernel of $B^{(q)}_{\Phi_\eta}$ with respect to $(\,\cdot\,|\,\cdot\,)_{\Phi_\eta}$. We take the Hermitian metric $\langle\,\cdot\,|\,\cdot\,\rangle$ on $T^{1,0}\mathbb C^n$ the holomorphic tangent bundle on $\mathbb C^n$ so that
$\set{\frac{\pr}{\pr z_j};\, j=1,\ldots,n}$ is an orthonormal basis. Let $M_{\Phi_\eta}: T_z^{1,0}\mathbb C^{n}\rightarrow T_z^{1,0}\mathbb C^{n}, z\in\mathbb C^{n}$ be the linear map defined by
\[\langle\,M_{\Phi_\eta}U\,|\,V\,\rangle=2\partial\overline\partial\Phi_{\eta}(U, \overline V), U, V\in T_z^{1,0}\mathbb C^{n}\]
and put
\[\hat{\mathbb R}_0=\{\eta\in\mathbb R;\, M_{\Phi_{\eta}}~\text{ has $n$ positive eigenvalues}\}.\]
The following lemma is known (see~\cite{HM12} and~\cite{MM07} ).

\begin{lemma}\label{abd}
If $\eta\not\in\hat{\mathbb R}_0$, then $B_{\Phi_\eta}(z, z)=0$ for all $z\in\mathbb C^{n}.$ If $\eta\in\hat{\mathbb R}_0$, then
\begin{equation}\label{e-gue180615q}
 B_{\Phi_\eta}(z,z)=e^{2\Phi_\eta(z)}(2\pi)^{-n}|\det M_{\Phi_\eta}|\cdot 1_{\hat{\mathbb R}_0}(\eta).
\end{equation}
\end{lemma}

Put
\[H^0_b(H_{n+1},\Phi_0):=\set{u\in L^2(H_{n+1}, \Phi_0);\, \ol U_{j,  H_{n+1}}u=0,\ \ \forall j=1,\ldots,n}\]
and define
\begin{equation}\label{e-gue180615e}
H^0_{b,[0,1]}(H_{n+1},\Phi_0):=\{u\in H^0_b(H_{n+1},\Phi_0);\, \mbox{$\hat u(z,\eta)=0$ for almost every $\eta\not\in[0,1]$}\}.
\end{equation}

The following is known (see Theorem 3.1, Lemma 3.5 in \cite{HM12} and Proposition 2.9 in~\cite{HL15})

\begin{theorem}\label{ii}
Let $u\in H^0_{b,[0,1]}(H_{n+1},\Phi_0)\bigcap C^\infty(H_{n+1})$. Then, for almost all $\eta\in\Real$, $\hat u(z,\eta)\in C^\infty(\Complex^{n})$, $\int_{\Complex^{n}}\abs{\hat u(z,\eta)}^2e^{-2\Phi_0(z)}dv(z)<\infty$,
\[
|\hat u(z,\eta)|^2\leq B_{\Phi_{\eta}}(z,z)\int_{\mathbb C^{n}}|\hat u(w,\eta)|^2e^{-2\Phi_0(w)}dv(w),
\]
$z\To\int_{\eta\in[0,1]}\hat u(z,\eta)d\eta$ is a continuous function and
\begin{equation}\label{e-gue180615qI}
u(0,0)=\frac{1}{2\pi}\int_{\eta\in[0,1]}\hat u(0,\eta)d\eta.
\end{equation}
\end{theorem}

Put
\[S_{H_{n+1},[0,1]}(x)=\sup\set{\abs{u(x)}^2;\, \norm{u}^2_{\Phi_0}=1, u\in H^0_{b,[0,1]}(H_{n+1},\Phi_0)]\bigcap C^\infty(H_{n+1})}. \]
We can now prove the following estimate for the Szeg\H{o} kernel.

\begin{theorem}\label{t-gue180615y}
We have
\[S_{H_{n+1},[0,1]}(0)\leq(2\pi)^{-n-1}\int_{\eta\in[0,1]\bigcap\hat{\mathbb R}_0}\abs{\det M_{\Phi_\eta}}d\eta.\]
\end{theorem}

\begin{proof}
By Theorem~\ref{ii} and notice that (see \eqref{s4-e12-1})
\[\int_{\eta\in[0,1]}|\hat u(w,\eta)|^2e^{-2\Phi_0(w)}dv(w)d\eta\leq \int|\hat u(w,\eta)|^2e^{-2\Phi_0(w)}dv(w)d\eta\leq 2\pi,\]
we have
\begin{equation}\label{e-gue150306ab}
\begin{split}
&\abs{u(0,0)}=\frac{1}{2\pi}\abs{\int_{\eta\in[0,1]}\hat u(0,\eta)d\eta}\\
&\leq\frac{1}{2\pi}\int_{\eta\in[0,1]}|\hat u(0,\eta)|\frac{(\int_{\mathbb C^{n}}|\hat u(w,\eta)|^2e^{-2\Phi_0(w)}dv(w))^{\frac12}}{(\int_{\mathbb C^{n}}|\hat u(w,\eta)|^2e^{-2\Phi_0(w)}dv(w))^{\frac12}}d\eta\\
\leq&\frac{1}{2\pi}\left(\int_{\eta\in[0,1]}\frac{|\hat u(0,\eta)|^2}{\int_{\mathbb C^{n}}|\hat u(w,\eta)|^2e^{-2\Phi_0(w)}dv(w)}d\eta\right)^{\frac12}\times\\
&\left(\int_{\eta\in[0,1]}|\hat u(w,\eta)|^2e^{-2\Phi_0(w)}dv(w)d\eta\right)^{\frac12}\\
%&\leq\sqrt{2\pi}\left(\int_{|\eta|\leq\delta}\frac{|\hat u_J(0,\eta)|^2}{\int_{\mathbb C^{n-1}}|\hat u(w,\eta)|^2e^{-\Phi_0(w)}dv(w)}d\eta\right)^{\frac12}\\
&\leq\frac{1}{\sqrt{2\pi}}\left(\int_{\eta\in[0,1]}\frac{B_{\Phi_{\eta}}(0,0)\int_{\mathbb C^{n}}|\hat u(w,\eta)|^2e^{-2\Phi_0(w)}dv(w)}{\int_{\mathbb C^{n}}|\hat u(w,\eta)|^2e^{-2\Phi_0(w)}dv(w)}d\eta\right)^{\frac12}\\
&\leq\frac{1}{\sqrt{2\pi}}\left(\int_{\eta\in[0,1]}B_{\Phi_\eta}(0,0)d\eta\right)^{\frac12}\\
&\leq(2\pi)^{-\frac{n+1}{2}}\left(\int_{\eta\in[0,1]\bigcap\hat{\mathbb R}_0}\abs{\det M_{\Phi_\eta}}d\eta\right)^{\frac12}.
\end{split}
\end{equation}
From \eqref{e-gue150306ab}, the theorem follows.
\end{proof}

\subsection{Szeg\H{o} kernel asymptotics}\label{s-gue180615rII}.

We now return to our situation. For $x\in\hat X$, put
\begin{equation}\label{e-gue180618}
\Real_{x,0}=\{t\in\Real;\, \mbox{$R^L_x+2t\mathcal{L}_x$ has $n$ positive eigenvalues}\}.
\end{equation}
The goal of this section, is to prove the following

\begin{theorem}\label{t-gue180618}
For every $x\in\hat X$, we have
\begin{equation}\label{e-gue180618I}
\limsup_{k\To+\infty}k^{-(n+1)}\hat S_k(x)\leq(2\pi)^{-n-1}\int_{t\in[0,1]\bigcap\Real_{x,0}}\abs{\det\bigr(R^L_x+2t\mathcal{L}_x\bigr)}dt.
\end{equation}
\end{theorem}

Fix $p\in\hat X$. Let $(x_1,\cdots,x_{2n+1})=(z,\theta)=(z_1,\cdots,z_{n},\theta), z_j=x_{2j-1}+ix_{2j},j=1,\cdots,n, x_{2n+1}=\theta$, and local rigid CR trivialization section $s$, $\abs{s}^2_{h^L}=e^{-2\Phi}$, be as in \eqref{e-can} defined in some small neighborhood $D$. We will use the same notations as in Section~\ref{s-gue180615rI}.  From \eqref{e-gue180615p}, there exists a sequence $u_{k_j}\in H^0_{b, 0\leq\alpha\leq k_j}(\hat X, L^{k_j})$, $0<k_1<k_2<\cdots$, such that
$\|u_{k_j}\|^2_{k_j}=1$ and
\begin{equation}\label{bbb}
\lim_{j\rightarrow\infty}k_j^{-n-1}|u_{k_j}(p)|^2_{h^{L^{k_j}}}=\limsup_{k\rightarrow\infty}k^{-n-1}\hat S_{k}(p).
\end{equation}
Put
$u_{k_j}=\tilde{u}_{k_j}\otimes s^{k_j}$, $\tilde{u}_{k_j}\in C^\infty(D)$. We will always use  $u_{k_j}$ to denote $\tilde u_{k_j}$ if there is no misunderstanding.
Write
\begin{equation}\label{r}
u_{k_j}=\sum\limits_{\alpha\in{\rm Spec\,}(-iT), 0\leq\alpha\leq k_j}u_{k_j,\alpha},\ \ Tu_{k_j,\alpha}=i\alpha u_{k_j,\alpha}
\end{equation}
and set
\begin{equation}
u_{(k_j)}=k^{-\frac{n+1}{2}}_j\sum\limits_{\alpha\in{ Spec\,}(-iT), 0\leq\alpha\leq k_j}F_{k_j}^\ast(u_{k_j,\alpha}).
\end{equation}
We can check that
\begin{equation}\label{u}
\|u_{(k_j)}\|^2_{k_jF_{k_j}^\ast\Phi, D_{\log k_j}}\leq\|u_{k_j}\|^2_{k_j}\leq1
\end{equation}
 and
\begin{equation}\label{ddd}
\ddbar_{b,(k_j)}u_{(k_j)}=0
\end{equation}
hold for all \(j\).
By Proposition~\ref{p-gue180615} and combining \eqref{u}, \eqref{ddd}, we deduce that for any $r>0$ with $D_{2r}\Subset D_{\log k_j}$ and every $s\in\mathbb N$,  there is a constant $C_{r,s}>0$ such that for every $j$ we have
\begin{equation}\label{x}
\|u_{(k_j)}\|^2_{k_jF^\ast_{k_j}\Phi, s+1, D_r}\leq C_{r,s}\left(1+\left\|\left(\frac{\partial}{\partial\theta}\right)^{s+1}u_{(k_j)}\right\|_{k_jF_{k_j}^\ast\Phi, D_{2r}}^2\right).
\end{equation}
Since
\begin{equation}\label{e-gue180618mp}
\begin{split}
\frac{\partial}{\partial\theta} u_{(k_j)}&=k^{-\frac{n+1}{2}}_j\frac{\partial}{\partial\theta}
\sum\limits_{\alpha\in{\rm Spec\,}(-iT), 0\leq\alpha\leq k_j} F_{k_j}^\ast(u_{k_j,\alpha})\\
&=k^{-\frac{n+1}{2}}_j\sum\limits_{\alpha\in{\rm Spec\,}(-iT), 0\leq\alpha\leq k_j}\left(\frac{i\alpha}{k_j}\right) u_{k_j,\alpha}\left(\frac{z}{\sqrt k_j}, \frac{\theta}{k_j}\right),
\end{split}
\end{equation}
we have
\begin{equation}\label{e-gue180618mpI}
\left(\frac{\partial}{\partial\theta}\right)^{s+1} u_{(k_j)}=k^{-\frac{n+1}{2}}_j\sum\limits_{\alpha\in{\rm Spec\,}(-iT), 0\leq\alpha\leq k_j}\left(\frac{i\alpha}{k_j}\right)^{s+1}u_{k_j,\alpha}\left(\frac{z}{\sqrt k_j}, \frac{\theta}{k_j}\right).
\end{equation}
Thus,
\begin{equation}\label{e-ftaI}
\begin{split}
\left\|\left(\frac{\partial}{\partial\theta}\right)^{s+1} u_{(k_j)}\right\|^2_{k_jF_{k_j}^\ast\Phi, D_r}
\leq k_j\sum\limits_{\alpha\in{\rm Spec\,}(-iT), 0\leq\alpha\leq k_j}\left\|k_j^{-\frac{n+1}{2}} u_{k_j,\alpha}\left(\frac{z}{\sqrt k_j}, \frac{\theta}{k_j}\right)\right\|^2_{k_jF_{k_j}^\ast\Phi, D_{2r}}.
\end{split}
\end{equation}
Since $T=\frac{\pr}{\pr\theta}$ on $D$, for each $j$, there is a function $\hat{u}_{k_j,\alpha}(z)\in C^\infty(D)$ such that
\begin{equation}
u_{k_j,\alpha}(z,\theta)=\hat{u}_{k_j,\alpha}(z)e^{i\alpha\theta}~\text{on}~D.
\end{equation}
For $r>0$, put $\tilde D_r:=\set{z=(z_1,\ldots,z_n)\in\mathbb C^n;\, \abs{z_j}<r, j=1,\ldots,n}$.
We have
\begin{equation}\label{e-fta}
\begin{split}
&k_j\sum\limits_{\alpha\in{\rm Spec\,}(-iT), 0\leq\alpha\leq k_j}\left\|k^{-\frac{n+1}{2}}_ju_{k_j,\alpha}\left(\frac{z}{\sqrt k_j}, \frac{\theta}{k_j}\right)\right\|^2_{k_jF_{k_j}^\ast\Phi, D_{2r}}\\
&\leq\sum\limits_{\alpha\in{\rm Spec\,}(-iT), 0\leq\alpha\leq k_j}\int_{D_{2r}}k^{-n}_j\left|\hat{u}_{k_j,\alpha}\left(\frac{z}{\sqrt k_j}\right)\right|^2e^{-2k_j\Phi(\frac{z}{\sqrt k_j})}m(\frac{z}{\sqrt k_j})dv(z)d\theta\\
&\leq\sum\limits_{\alpha\in{\rm Spec\,}(-iT), 0\leq\alpha\leq k_j}(4r)\int_{\tilde D_{\frac{2r}{\sqrt k_j}}}|\hat{u}_{k_j,\alpha}(z)|^2e^{-2k_j\Phi(z)}m(z)dv(z)\\
&\leq\frac{4r}{\varepsilon}\sum\limits_{\alpha\in{\rm Spec\,}(-iT), 0\leq\alpha\leq k_j}\int_{\tilde D_{\frac{2r}{\sqrt k_j}}}\int_{ |\theta|<\varepsilon}|u_{k_j,\alpha}(z,\theta)|^2e^{-2k_j\Phi(z)}m(z,\theta)dv(z)d\theta\\
&\leq \frac{4r}{\varepsilon}\sum\limits_{\alpha\in{\rm Spec\,}(-iT), 0\leq\alpha\leq k_j}\| u_{k_j,\alpha}\|^2_{k_j}
\leq \frac{4r}{\varepsilon}\|u_{k_j}\|^2_{k_j}\leq\frac{4r}{\varepsilon},
\end{split}
\end{equation}
where $\varepsilon>0$ is a small constant,  $dv_{\hat X}=m(z)dv(z)d\theta$, $m(z)\in C^\infty(D)$, $dv(z)=2^ndx_1\cdots dx_{2n}$.  From (\ref{e-fta}) and (\ref{e-ftaI}), we deduce that
\begin{equation}\label{Eq:DerivTheta}
\left\|\left(\frac{\partial}{\partial\theta}\right)^{s+1}u_{(k_j)}\right\|^2_{kF_k^\ast\Phi, D_r}\leq \tilde C_{r,s},\ \ \forall j,
\end{equation}
where $\tilde C_{r, s}$ is a constant independent of $j$. Combining \eqref{Eq:DerivTheta} with (\ref{x}), there exists a constant $C_{r,s}^\prime>0$ independent of $j$ such that
\begin{equation}\label{y}
\|u_{(k_j)}\|^2_{k_jF_{k_j}^\ast\Phi, s+1, D_r}\leq C_{r,s}^\prime
\end{equation}
holds for all \(j\).
From \eqref{y}, we can use the same argument as in the proof of Theorem~2.9 in \cite{HM12} and conclude that there is a subsequence $\{u_{(k_{v_1})},u_{(k_{v_2})},\ldots\}$ of $\{u_{(k_j)}\}$, $0<k_{v_1}<k_{v_2}<\cdots$, such that $u_{(k_{v_\ell})}$ converges uniformly with all derivatives on any compact subset of $H_{n+1}$ to a smooth function $u\in C^\infty(H_{n+1})$ as $\ell\To\infty$,
\begin{equation}\label{e-gue180619}
\limsup_{k\To\infty}k^{-n-1}\hat S_{k}(p)=\abs{u(0)}^2
\end{equation}
and
\begin{equation}\label{e-gue180619I}
\|u\|_{\Phi_0}\leq 1,\ \  \ddbar_{b, H_{n+1}}u=0.
\end{equation}

\begin{proof}[Proof of Theroem~\ref{t-gue180618}]
We are now ready to prove Theorem~\ref{t-gue180618}. We will use the same notations as before. We first claim that

\begin{equation}\label{e-gue180619mp}
\mbox{$\hat u(z,\eta)=0$ for almost every $\eta\notin[0,1]$,}
\end{equation}
where $\hat u(z, \eta)$ is the partial Fourier transform of $u(z, \theta)$ with respect to $\theta$ (see the discussion after \eqref{e-gue180615c}). To prove the claim
\eqref{e-gue180619mp}, we only need to show that for any $\varphi(z,\eta)\in C_0^\infty(\mathbb C^{n}\times\{\eta\in\mathbb R;\, \eta\notin[0,1]\})$, we have
\begin{equation}\label{e-gue180619mpI}
\int_{H_{n+1}}\hat u(z,\eta)\varphi(z,\eta)e^{-2\Phi_0(z)}dv(z)d\eta=0.
\end{equation}
We assume ${\rm Supp\,}\varphi\Subset \mathbb{D}^n(r_0)\times\{\eta\in\mathbb R;\, \eta\notin[0,1]\}$. Here, \(\mathbb{D}^n(r_0)=\{z\in\mathbb C^{n};\, |z_j|\leq r_0,\, j=1,\cdots,n\}\) denotes the polydisc around the origin of common radius \(r_0\). %$\abs{z}\leq r_0$ means that $|z_j|<r_0$, $j=1,\cdots,n$.
 Choose $\chi\in C_0^\infty(\mathbb R)$ such that $\chi\equiv1$ when $|\theta|\leq 1$ and ${\rm Supp\,}\chi\Subset\{\theta\in\mathbb R;\, |\theta|<2\}$. From \eqref{s4-e12-2}, we have
\begin{equation}\label{e-gue180622I}
\begin{split}
\int_{H_{n+1}}\hat u(z,\eta)&\varphi(z,\eta)e^{-2\Phi_0(z)}dv(z)d\eta
=\int_{H_{n+1}}u(z,\theta)\hat{\varphi}(z,\theta)e^{-2\Phi_0(z)}dv(z)d\theta\\
&=\lim_{r\rightarrow\infty}\int_{H_{n+1}}u(z,\theta)\hat{\varphi}(z,\theta)e^{-2\Phi_0(z)}
\chi\left(\frac{\theta}{r}\right)dv(z)d\theta,
\end{split}
\end{equation}
where $\hat{\varphi}(z,\theta):=\int_{\mathbb R} e^{-i\theta\eta}\varphi(z,\eta)d\eta$ is the partial Fourier transform of $\varphi(z,\eta)$ respect to $\eta$. For simplicity, we may assume that $u_{(k_{j})}$ converges uniformly with all derivatives on any compact subset of $H_{n+1}$ to $u$ as $j\To\infty$.
As before, on $D$, for each $j$, we can write
\[\begin{split}
&u_{k_j}=\sum\limits_{\alpha\in{\rm Spec\,}(-iT), 0\leq\alpha\leq k_j} u_{k_j,\alpha},\  \ Tu_{k_j,\alpha}=i\alpha u_{k_j,\alpha},\\
&u_{k_j,\alpha}(z,\theta)=\hat{u}_{k_j,\alpha}(z)e^{i\alpha\theta}~\text{on}~D,\ \ \hat{u}_{k_j,\alpha}(z)\in C^\infty(D),\ \ \forall j,\ \ \forall\alpha\in{\rm Spec\,}(-iT).
\end{split}\]
When $r$ is fixed, by dominated convergence theorem we find
\begin{equation}\label{e-gue150303b}
\begin{split}
~~~~&\int_{H_{n+1}}u(z,\theta)\hat{\varphi}(z,\theta)e^{-2\Phi_0(z)}
\chi\left(\frac{\theta}{r}\right)dv(z)d\theta\\
&=\lim_{j\rightarrow\infty}\sum_{\alpha\in{\rm Spec\,}(-iT), 0\leq\alpha\leq k_j}\int_{H_{n+1}}k_j^{-\frac{n+1}{2}}
\hat u_{k_j,\alpha}\left(\frac{z}{\sqrt {k_j}}\right)e^{i\frac{\alpha}{ k_j}\theta}\hat\varphi(z,\theta)\chi\left(\frac{\theta}{r}\right)e^{-2\Phi_0(z)}dv(z)d\theta\\
&=\lim_{j\rightarrow\infty}\sum_{\alpha\in{\rm Spec\,}(-iT), 0\leq\alpha\leq k_j}\int_{\mathbb{D}^n(r_0)}\int_{\mathbb R}
k_j^{-\frac{n+1}{2}}\hat u_{k_j,\alpha}\left(\frac{z}{\sqrt {k_j}}\right)e^{i\frac{\alpha}{k_j}\theta}
\hat\varphi(z,\theta)\chi\left(\frac{\theta}{r}\right)e^{-2\Phi_0(z)}dv(z)d\theta.
\end{split}
\end{equation}
Since ${\rm Supp\,}\varphi(z,\eta)\Subset\mathbb{D}^n(r_0)\times\{\theta\in\mathbb R: \eta\notin[0,1]\}$ and $0\leq\frac{\alpha}{k_j}\leq1$, we have
\begin{equation}\label{ff}
\begin{split}
&\sum_{\alpha\in{\rm Spec\,}(-iT), 0\leq\alpha\leq k_j}\int_{\mathbb{D}^n(r_0)}\int_{\Real}k_j^{-\frac{n+1}{2}}\hat u_{k_j,\alpha}\left(\frac{z}{\sqrt{k_j}}\right)e^{i\frac{\alpha}{k_j}\theta}\hat\varphi(z,\theta)e^{-2\Phi_0(z)}dv(z)d\theta\\
&=(2\pi)\sum_{\alpha\in{\rm Spec\,}(-iT), 0\leq\alpha\leq k_j}\int_{\mathbb{D}^n(r_0)}\int_{\Real}k_j^{-\frac{n+1}{2}}\hat u_{k_j,\alpha}\left(\frac{z}{\sqrt{k_j}}\right)\varphi(z,\frac{\alpha}{k_j})e^{-2\Phi_0(z)}dv(z)=0.\,\,\,\,\,\,
\end{split}
\end{equation}
By \eqref{ff}, H\"older inequality and some straightforward calculation, we have
\begin{equation}\label{e-gue150303bI}
\begin{split}
&\abs{\sum_{\alpha\in{\rm Spec\,}(-iT), 0\leq\alpha\leq k_j}\int_{\mathbb{D}^n(r_0)}\int_{\mathbb R}
k_j^{-\frac{n+1}{2}}\hat u_{k_j,\alpha}\left(\frac{z}{\sqrt {k_j}}\right)e^{i\frac{\alpha}{k_j}\theta}
\hat\varphi(z,\theta)\chi\left(\frac{\theta}{r}\right)e^{-2\Phi_0(z)}dv(z)d\theta}\\
&=\abs{\sum_{\alpha\in{\rm Spec\,}(-iT), 0\leq\alpha\leq k_j}\int_{\mathbb{D}^n(r_0)}\int_{\mathbb R}k_j^{-\frac{n+1}{2}}\hat u_{k_j,\alpha}\left(\frac{z}{\sqrt{k_j}}\right)e^{i\frac{\alpha}{k_j}\theta}
\hat\varphi(z,\theta)\left(\chi\left(\frac{\theta}{r}\right)-1\right)e^{-2\Phi_0(z)}dv(z)d\theta}\\
&\leq\sum_{\alpha\in{\rm Spec\,}(-iT), 0\leq\alpha\leq k_j}\left(\int_{\mathbb{D}^n(r_0)}\int_{|\theta|\geq
r}k_j^{-n-1}|\hat u_{k_j, \alpha}\left(\frac{z}{\sqrt{k_j}}\right)|^2\cdot|\hat\varphi(z,\theta)|e^{-2\Phi_0(z)}dv(z)d\theta
\right)^{\frac12}\times\\
&\left(\int_{\mathbb{D}^n(r_0)}\int_{|\theta|\geq r}|\hat\varphi(z,\theta)|e^{-2\Phi_0(z)}dv(z)d\theta\right)^{\frac12}\\
&\leq C_0\sum_{\alpha\in{\rm Spec\,}(-iT), 0\leq\alpha\leq k_j}\frac{1}{\sqrt{k_j}}\norm{u_{k_j,\alpha}}\left(\int_{\mathbb{D}^n(r_0)}\int_{|\theta|\geq r}|\hat\varphi(z,\theta)|e^{-2\Phi_0(z)}dv(z)d\theta\right)^{\frac12}\\
&\leq C_1\norm{u_{k_j}}^2_{k_j}
\int_{\mathbb{D}^n(r_0)}\int_{|\theta|\geq r}|\hat\varphi(z,\theta)|e^{-2\Phi_0(z)}dv(z)d\theta,
\end{split}
\end{equation}
where $C_0>0$, $C_1>0$ are constants independent of $j$. From \eqref{e-gue150303b} and \eqref{e-gue150303bI}, we deduce
\begin{equation}\label{e-gue180622}
\begin{split}
\abs{\int_{H_{n+1}}u(z,\theta)\hat{\varphi}(z,\theta)e^{-2\Phi_0(z)}
\chi(\frac{\theta}{r})dv(z)d\theta}\leq C_2\int_{\mathbb{D}^n(r_0)}\int_{|\theta|\geq r}|\hat\varphi(z,\theta)|e^{-2\Phi_0(z)}dv(z)d\theta,
\end{split}
\end{equation}
where $C_2>0$ is a constant independent of $k_j$. From \eqref{e-gue180622}, \eqref{e-gue180622I} and notice that
\[\lim_{r\To+\infty}\int_{\mathbb{D}^n(r_0)}\int_{|\theta|\geq r}|\hat\varphi(z,\theta)|e^{-2\Phi_0(z)}dv(z)d\theta=0,\]
we get \eqref{e-gue180619mpI} and the claim \eqref{e-gue180619mp} follows.

From \eqref{e-gue180619}, \eqref{e-gue180619I}, \eqref{e-gue180619mp} and Theorem~\ref{t-gue180615y}, we find
\begin{equation}\label{e-gue180622a}
\limsup_{k\To+\infty}k^{-n-1}\hat S_k(p)\leq S_{H_{n+1},[0,1]}(0)\leq(2\pi)^{-n-1}\int_{\eta\in[0,1]\bigcap\hat{\mathbb R}_0}\abs{\det M_{\Phi_\eta}}d\eta.
\end{equation}
It is not difficult to check that
\[\int_{\eta\in[0,1]\bigcap\hat{\mathbb R}_0}\abs{\det M_{\Phi_\eta}}d\eta
=\int_{t\in[0,1]\bigcap\Real_{p,0}}\abs{\det\bigr(R^L_p+2t\mathcal{L}_p\bigr)}dt.\]
From this observation and \eqref{e-gue180622a}, the theorem follows.
\end{proof}

\section{Proofs of Theorem~\ref{t-gue180528} and Theorem~\ref{t-gue180528I}}\label{s-gue180606m}

In this section, we will prove Theorem~\ref{t-gue180528} and Theorem~\ref{t-gue180528I}. We will use the same notations and assumptions as in Section~\ref{s-gue170805}.
Let $M$ be a compact complex manifold of complex dimension $n$ and let $(L,h^L)$ be a holomorphic line bundle over $M$, where $h^L$ denotes the given Hermitian fiber metric on $L$. Let $R^L$ be the curvature on $M$ induced by $h^L$. We assume that $R^L$ is positive at every point of $M$. Consider a new CR manifold $\hat X:=M\times X$ with natural CR structure $T^{1,0}\hat X:=T^{1,0}M\oplus T^{1,0}X$, where $T^{1,0}M$ denotes the complex structure on $M$. Then, $(\hat X, T^{1,0}\hat X)$ is a compact CR manifold of dimension $2(2n)+1$.  Let $T=\mu_1T_1+\cdots+\mu_dT_d$ and $\omega_0$ be as in the discussion after \eqref{e-gue180528I}. From now on, we consider $T$ as a global non-vanishing vector field on $\hat X$ and $\omega_0$ as a global non-vanishing one form on $\hat X$. It is clear that
\[\Complex T\hat X:=T^{1,0}\hat X\oplus T^{0,1}\hat X\oplus\Complex T\ \ \mbox{on $\hat X$}.\]
With the one form $\omega_0$, we define the Levi form $\mathcal{L}_{\hat x}$ at $\hat x\in\hat X$ as Definition~\ref{d-gue150808}. The torus action $T^d$ acting on $X$ lifts to $\hat X$ in the natural way:
\[g\circ (z,x):=(z,g\circ x),\ \ \forall g\in T^d,\]
where $(z,x)\in M\times X$. Moreover, the global non-vanishing vector field on $T$ induces a transversal CR $\mathbb{R}$-action $\eta$ on $\hat X$.

Fix a Hermitian metric $\langle\,\cdot\,,\,\cdot\,\rangle$ on $\Complex TM$. The torus invariant Hermitian metric $\langle\,\cdot\,|\,\cdot\,\rangle$ on $\Complex TX$ and the Hermitian metric $\langle\,\cdot\,,\,\cdot\,\rangle$ on $\Complex TM$ induce a $\mathbb{R}$-invariant Hermitian metric $\langle\,\cdot\,|\,\cdot\,\rangle$ on $\Complex T\hat X$ and Assumption~\ref{a-gue180605} holds with such Hermitian metric. From now on, we fix the $\mathbb{R}$-invariant Hermitian metric $\langle\,\cdot\,|\,\cdot\,\rangle$ on $\Complex T\hat X$ induced by the Hermitian metric $\langle\,\cdot\,,\,\cdot\,\rangle$ on $\Complex TM$ and the torus invariant Hermitian metric $\langle\,\cdot\,|\,\cdot\,\rangle$ on $\Complex TX$. We consider $(L,h^L)$ as a CR line bundle over $\hat X$ such that $L$ is trivial on $X$. It is obvious that $L$ is a rigid CR line bundle and $h^L$ is a rigid Hermitian metric on $L$. Since $L$ is positive and $X$ is strongly pseudoconvex, we conclude that
\begin{equation}\label{e-gue180605p}
\mbox{$R^L_{\hat x}+2s\mathcal{L}_{\hat x}$ is positive definite at every $\hat x\in\hat X$, for every $s\in]0,+\infty[$}.
\end{equation}

The Hermitian metric on $L^k$ induced by $h^L$ is denoted by $h^{L^k}$. Let $(\,\cdot\,|\,\cdot\,)_k$ and $(\,\cdot\,,\,\cdot\,)_k$ be the $L^2$ inner products on
$C^\infty(\hat X,L^k)$ and $C^\infty(M,L^k)$ induced by $h^{L^k}$, $\langle\,\cdot\,|\,\cdot\,\rangle$ and $h^{L^k}$, $\langle\,\cdot\,,\,\cdot\,\rangle$ respectively and let $(\,\cdot\,|\,\cdot\,)$ be the $L^2$ inner product on $C^\infty(X)$ induced by $\langle\,\cdot\,|\,\cdot\,\rangle$.

Since the transversal CR $\mathbb{R}$-action $\eta$ on $\hat X$ comes from the torus action on $\hat X$, we can repeat the proof of Theorem 4.5 in~\cite{HHL17} and conclude that Assumption~\ref{a-gue180604} holds and
\begin{equation}\label{e-gue180608e}
{\rm Spec\,}(-iT)\subset\set{\mu_1p_1+\cdots+\mu_dp_d;\, (p_1,\ldots,p_d)\in\mathbb Z^d}.
\end{equation}
Take $\tau\in C^\infty_0(]0,+\infty[)$ and let $\hat S_{k,\tau}(\hat x)\in C^\infty(\hat X)$ be as in \eqref{e-gue180605}. From Theorem~\ref{t-gue180605}, we deduce that
\begin{equation}\label{e-gue180605mp}
\begin{split}
&\hat S_{k,\tau}(\hat x)\sim\sum^\infty_{j=0}a_j(\hat x)k^{2n+1-j}\text{ in }S^{2n+1}_{{\rm loc\,}}
(1;\hat X),\\
&a_j(\hat x)\in C^\infty(\hat X),\ \ j=0,1,2,\ldots,\\
&a_0(\hat x)=(2\pi)^{-2n-1}\int\abs{\det\bigr(R^L_{\hat x}+2t\mathcal{L}_{\hat x}\bigr)}\abs{\tau(t)}^2dt,\ \ \forall\hat x\in\hat X.
\end{split}
\end{equation}
For $\hat x=(z,x)\in\hat X=M\times X$, it is easy to see that
\begin{equation}\label{e-gue180605mpI}
\det\bigr(R^L_{\hat x}+2t\mathcal{L}_{\hat x}\bigr)=(2t)^n\Bigr(\det R^L_z\Bigr)\Bigr(\det\mathcal{L}_x\Bigr),
\end{equation}
where $\det R^L_z=v_1(z)\cdots v_n(z)$, $v_j(z)$, $j=1,\ldots,n$, are the eigenvalues of $R^L_z$ with respect to $\langle\,\cdot\,,\,\cdot\,\rangle$ and $\det\mathcal{L}_x$ is as in the discussion after \eqref{e-gue180528aI}. From \eqref{e-gue180605mp} and \eqref{e-gue180605mpI}, we have
\begin{equation}\label{e-gue180605mpII}
a_0(\hat x)=(2\pi)^{-2n-1}\abs{\det R^L_{z}}\abs{\det\mathcal{L}_x}\int (2t)^n\abs{\tau(t)}^2dt,\ \ \forall\hat x=(z,x)\in\hat X=M\times X.
\end{equation}

For $k\in\mathbb N$, let $H^0(M,L^k):=\set{u\in C^\infty(M,L^k);\, \ddbar u=0}$. Let $\set{g_1,\ldots,g_{d_k}}$ be an orthonormal basis for $H^0(M,L^k)$ with respect to
$(\,\cdot\,,\,\cdot\,)_k$ and put
\[B_k(z):=\sum^{d_k}_{j=1}\abs{g_j(z)}^2_{h^{L^k}}.\]
Fix $(p_1,\ldots,p_d)\in\mathbb Z^d$. Let $H^0_{b,p_1,\dots,p_d}(X)$ and $S_{p_1,\ldots,p_d}(x)\in C^\infty(X)$ be as in \eqref{e-gue180605a} and \eqref{e-gue180528II} respectively. We need

\begin{lemma}\label{l-gue180605a}
Fix $(p_1,\ldots,p_d)\in\mathbb Z^d$ and let $\alpha=p_1\mu_1+\cdots+p_d\mu_d\in{\rm Spec\,}(-iT)$. We have
\[H^0_{b,\alpha}(\hat X, L^k)={\rm span\,}\set{g(z)h(x);\, g(z)\in H^0(M,L^k), h(x)\in H^0_{b,p_1,\ldots,p_d}(X)}.\]
\end{lemma}

\begin{proof}
It is clear that
\[{\rm span\,}\set{g(z)h(x);\, g(z)\in H^0(M,L^k), h(x)\in H^0_{b,p_1,\ldots,p_d}(X)}\subset H^0_{b,\alpha}(\hat X, L^k).\]
Let $u=u(z,x)\in H^0_{b,\alpha}(\hat X,L^k)$. For every $x\in X$, it is clear that $v_x(z): z\To u(z,x)$ is an element in $H^0(M,L^k)$. Hence,
\begin{equation}\label{e-gue180605am}
u(z,x)=v_x(z)=\sum^{d_k}_{j=1}g_j(z)(\,u(z,x)\,,\,g_j(z)\,)_k,
\end{equation}
where $\set{g_1,\ldots,g_{d_k}}$ is an orthonormal basis for $H^0(M,L^k)$. Now, for each $j$, we can check that $x\To(\,u(z,x)\,,\,g_j(z)\,)_k\in H^0_{b,p_1,\ldots,p_d}(X)$. From this observation and \eqref{e-gue180605am}, the lemma follows.
\end{proof}

Let $S_{k,\tau}(x)\in C^\infty(X)$ be as in \eqref{e-gue180528a}. We have

\begin{theorem}\label{t-gue180605e}
With the notations and assumptions above, we have
\begin{equation}\label{e-gue180608eI}
\hat S_{k,\tau}(\hat x)=B_k(z)S_{k,\tau}(x),\ \ \forall \hat x=(z,x)\in\hat X=M\times X.
\end{equation}
\end{theorem}

\begin{proof}
Fix $(p_1,\ldots,p_d)\in\mathbb Z^d$. Let $\set{f_1,\ldots,f_d}$ be an orthonormal basis for $H^0_{b,p_1,\ldots,p_d}(X)$. From Lemma~\ref{l-gue180605a}, we have
\[\set{g_jf_\ell |\, j=1,\ldots,d_k, \ell=1,\ldots,d}\]
is an orthonormal basis for $H^0_{b,\alpha}(\hat X)$, where $\alpha=p_1\mu_1+\cdots+p_d\mu_d$ and
$\set{g_1,\ldots,g_{d_k}}$ is an orthonormal basis for $H^0(M,L^k)$. Hence,
\begin{equation}\label{e-gue180612}
\hat S_{k,\alpha}(\hat x)=B_k(z)S_{p_1,\ldots,p_d}(x),\ \ \forall \hat x=(z,x)\in\hat X=M\times X,
\end{equation}
where $\hat S_{k,\alpha}$ is as in \eqref{e-gue180604m}. From \eqref{e-gue180612}, the theorem follows.
\end{proof}

\begin{proof}[Proof Theorem~\ref{t-gue180528}]
From Bergman kernel asymptotic expansion for positive line bundles (see~\cite{DLM06}, ~\cite{HM12},~\cite{MM07},~\cite{Zel98}), we know that
\begin{equation}\label{e-gue180612I}
\begin{split}
&B_k(z)\sim\sum^n_{j=0}k^{n-j}b_j(z)\ \text{ in }S^{n}_{{\rm loc\,}}
(1;M),\\
&b_j(z)\in C^\infty(M),\ \ j=0,1,2,\ldots,\\
&b_0(z)=(2\pi)^{-n}\abs{\det R^L_z},\ \ \forall z\in M.
\end{split}
\end{equation}
Let $k_0\in\mathbb N$ be a large constant so that $B_k(z)\geq ck^n$ on $M$, for all $k\geq k_0$, where $c>0$ is a constant. Hence, for $k\geq k_0$,
\begin{equation}\label{e-gue180612II}
\begin{split}
&\frac{1}{B_k(z)}\sim\sum^n_{j=0}k^{-n-j}c_j(z)\ \text{ in }S^{-n}_{{\rm loc\,}}
(1;M),\\
&c_j(z)\in C^\infty(M),\ \ j=0,1,2,\ldots,\\
&c_0(z)=(2\pi)^{n}\abs{\det R^L_z}^{-1},\ \ \forall z\in M.
\end{split}
\end{equation}
From \eqref{e-gue180608eI}, for $k\geq k_0$, we have
\begin{equation}\label{e-gue180612III}
S_{k,\tau}(x)=\frac{\hat S_{k,\tau}(\hat x)}{B_k(z)},\ \ \forall \hat x=(z,x)\in\hat X=M\times X.
\end{equation}
From \eqref{e-gue180605mp}, \eqref{e-gue180605mpII}, \eqref{e-gue180612II} and \eqref{e-gue180612III}, the theorem follows.
\end{proof}

In the rest of this section, we will prove Theorem~\ref{t-gue180528I}. Let $S_k(x)\in C^\infty(X)$ be as in \eqref{e-gue180529mII}. We need

\begin{theorem}\label{t-gue180622mp}
With the notations and assumptions above, we have
\[\limsup_{k\To+\infty}k^{-(n+1)}S_k(x)\leq\frac{1}{2}\pi^{-n-1}\frac{1}{n+1}\abs{{\rm det\,}\mathcal{L}_x},\]
for every $x\in X$.
\end{theorem}

\begin{proof}
Let $\hat S_k(\hat x)\in C^\infty(\hat X)$ be as in \eqref{e-gue180612m}. From \eqref{e-gue180608eI}, we have
\begin{equation}\label{e-gue180622m}
\hat S_k(\hat x)=B_k(z)S_k(x),\ \ \forall \hat x=(z,x)\in\hat X=M\times X.
\end{equation}
By Theorem~\ref{t-gue180618}, we have
\begin{equation}\label{e-gue180622mII}
\limsup_{k\To+\infty}k^{-(2n+1)}\hat S_k(\hat x)\leq(2\pi)^{-2n-1}\int_{t\in[0,1]\bigcap\mathbb{R}_{\hat x,0}}\abs{\det(R^L_{\hat x}+2t\mathcal{L}_{\hat x})}dt,
\end{equation}
for every $\hat x=(z,x)\in\hat X=M\times X$. From \eqref{e-gue180622m} and \eqref{e-gue180612I}, we conclude that
\begin{equation}\label{e-gue180622mI}
\limsup_{k\To+\infty}k^{-(2n+1)}\hat S_k(\hat x)=(2\pi)^{-n}\abs{\det R^L_z}\limsup_{k\To+\infty}k^{-(n+1)}S_k(x),
\end{equation}
for every $\hat x=(z,x)\in\hat X=M\times X$. From \eqref{e-gue180622mII} and \eqref{e-gue180622mI}, we deduce that
\begin{equation}\label{e-gue180622mIII}
\limsup_{k\To+\infty}k^{-(n+1)}S_k(x)\leq (2\pi)^{-n-1}\Bigr(\abs{\det R^L_z}\Bigr)^{-1}\int_{t\in[0,1]\bigcap\mathbb{R}_{\hat x,0}}\abs{\det(R^L_{\hat x}+2t\mathcal{L}_{\hat x})}dt,
\end{equation}
for every $\hat x=(z,x)\in\hat X=M\times X$. It is easy to check that
\begin{equation}\label{e-gue180622p}
\begin{split}
\int_{t\in[0,1]\bigcap\mathbb{R}_{\hat x,0}}\abs{\det(R^L_{\hat x}+2t\mathcal{L}_{\hat x})}dt&=\abs{\det R^L_z}\abs{\det\mathcal{L}_x}\int^1_0(2t)^ndt\\
&=\abs{\det R^L_z}\abs{\det\mathcal{L}_x}2^n\frac{1}{n+1},
\end{split}
\end{equation}
for every $\hat x=(z,x)\in\hat X=M\times X$. From \eqref{e-gue180622p} and \eqref{e-gue180622mIII}, the theorem follows.
\end{proof}

\begin{theorem}\label{t-gue180622mpI}
With the notations and assumptions above, we have
\[\liminf_{k\To+\infty}k^{-(n+1)}S_k(x)\geq\frac{1}{2}\pi^{-n-1}\frac{1}{n+1}\abs{{\rm det\,}\mathcal{L}_x},\]
for every $x\in X$.
\end{theorem}

\begin{proof}
Let $\frac{1}{2}>\delta>0$ be a small constant and let $\tau_\delta\in C^\infty(]0,1[)$, $0\leq\tau_\delta\leq1$, $\tau_\delta=1$ on $[\delta,1-\delta]$. From Theorem~\ref{t-gue180528}, \eqref{e-gue180529m} and notice that $S_k(x)\geq S_{k,\tau_\delta}(x)$, for every $x\in X$ and every $\frac{1}{2}>\delta>0$, we have
\begin{equation}\label{e-gue180622y}
\begin{split}
\liminf_{k\To+\infty}k^{-(n+1)}S_k(x)&\geq\liminf_{k\To+\infty}k^{-(n+1)}S_{k,\tau_\delta}(x)\\
&\geq(2\pi)^{-n-1}\abs{{\rm det\,}\mathcal{L}_x}\int (2t)^n\abs{\tau_\delta(t)}^2dt\\
&\geq (2\pi)^{-n-1}\abs{{\rm det\,}\mathcal{L}_x}\int^{1-\delta}_\delta (2t)^ndt\\
&=\frac{1}{2}\pi^{-n-1}\frac{1}{n+1}\abs{{\rm det\,}\mathcal{L}_x}\Bigr((1-\delta)^{n+1}-\delta^{n+1}\Bigr),
\end{split}
\end{equation}
for every $x\in X$ and every  $\frac{1}{2}>\delta>0$. Let $\delta\To0$ in \eqref{e-gue180622y}, the theorem follows.
\end{proof}

\begin{proof}[Proof of Theorem~\ref{t-gue180528I}]
From Theorem~\ref{t-gue180622mp} and Theorem~\ref{t-gue180622mpI}, Theorem~\ref{t-gue180528I} follows.
\end{proof}


\begin{thebibliography}{10}

\bibitem{DLM06}
X.~Dai, K.~Liu and X.~Ma, \emph{On the asymptotic expansion of {B}ergman kernel}, J. Differential Geom. \textbf{72} (2006), no.~1, 1--41.
\bibitem{HM09}
C-Y.~Hsiao and G.~Marinescu,
\emph{Szeg\H{o} kernel asymptotics and Morse inequalities on CR manifolds},
Math.\ Z.\ \textbf{271} (2012), 509--553.
\bibitem{HL15}
C-Y.~Hsiao and X.~Li,
\emph{Szeg\"o kernel asymptotics and Morse inequalities on CR manifolds with $S^1$ action},
 available at preprint  arXiv:1502.02365, to appear in the Asian Journal of Mathematics.
\bibitem{HHL17} H.~Herrmann, C.-Y.~Hsiao, X.~Li,
	\emph{Szeg\H{o} kernels and equivariant embedding theorems for CR manifolds},  available at preprint arXiv:1710.04910.
	
\bibitem{HM12}
C-Y.~Hsiao and G.~Marinescu,
\emph{Asymptotics of spectral function of lower energy forms and {Bergman} kernel of
semi-positive and big line bundles}, Comm.\ Anal.\ Geom.\ \textbf{22} (2014), 1--108.
\bibitem{MM07}
X.~Ma and G.~Marinescu, \emph{Holomorphic {Morse} inequalities and {Bergman} kernels}, Progress
  in Math., vol. 254, Birkh{\"a}user, Basel, 2007, 422 pp.
\bibitem{Zel98}
S.~Zelditch, \emph{Szeg\"{o} kernels and a theorem of {Tian}},
Internat.\ Math.\ Res.\ Notices., (1998), no.~6, 317--331.


\end{thebibliography}
\end{document}